\documentclass[]{amsart}

\usepackage{amsthm}
\newcommand{\norm}[1]{\left\Vert#1\right\Vert}
\newcommand{\abs}[1]{\left\vert#1\right\vert}
\newcommand{\RR}{\mathbb R}

\newcommand{\F}{\mathcal{F}}
\newtheorem{theorem}{Theorem}
\newtheorem{lemma}{Lemma}
\newtheorem{proposition}{Proposition}

\theoremstyle{remark}
\newtheorem{remark}{Remark}
\newtheorem{example}{Example}
\theoremstyle{definition}
\newtheorem{definition}{Definition}

\begin{document}
\thispagestyle{plain}
\newcounter{N}


\author{Y. Kozachenko$^{\rm a}$,
A. Melnikov$^{\rm b\ast}$ and Y. Mishura$^{\rm a}$}

\thanks{$^\ast$Corresponding author. Email: melnikov@ualberta.ca
\vspace{6pt}\\
$^{\rm a}${\em{Department of Probability, Statistics
and Actuarial Mathematics, Mechanics and Mathematics Faculty, Kyiv Taras Shevchenko National University,
Volodymyrska, 60, 01601 Kyiv}}; $^{\rm b}${\em{Department of Mathematical and
Statistical Sciences,
University of Alberta,
632 Central Academic Building,
Edmonton, AB T6G 2G1, Canada}}}

\title{On  drift parameter estimation in models with fractional Brownian motion}

\maketitle

\begin{abstract}
We consider a stochastic differential equation involving standard and fractional Brownian motion  with unknown drift parameter to be estimated. We investigate the standard maximum likelihood estimate  of the drift parameter,  two non-standard estimates and three estimates for the sequential estimation.  Model strong consistency and some other properties are proved. The linear model  and Ornstein-Uhlenbeck model  are studied in detail. As an auxiliary result, an asymptotic behavior of the fractional derivative of the fractional Brownian motion is established.
\end{abstract}

\section{Introduction}
Modern mathematical statistics tends to shift away from the standard statistical schemes based on  independent random variables; besides, these days many statistical models are based on continuous time. Therefore, the corresponding statistical problems (e.g., parameter estimation) can be handled by methods of the theory of stochastic processes in addition to the standard statistical methods. Statistics for stochastic processes is  well-developed for diffusion processes and even for semimartingales (see, for instance, \cite{lipshir}) but is still developing for the processes with long-range dependence. The latter is an integral part of stochastic processes, featuring a wide spectrum of applications applications in economics, physics, finance and other fields.
The present paper is devoted to the parameter estimation in such models involving fractional Brownian motion (fBm) with Hurst parameter $H>\frac12$ which  is a well-known long-memory process. The paper also studies a mixed model based on both standard and fractional Brownian motion which turns out to be more flexible. One of the reasons to consider such  model  comes from the modern mathematical finance  where it it has become very popular to assume that
the underlying random noise consists
of two parts: the  fundamental  part, describing the economical background for the
stock price, and the trading  part, related to the randomness inherent to the
stock market. In our case the fundamental part of the noise has a long
memory  while the trading part is  a white noise.

 Statistical aspects of   models involving fractional Brownian motion were studied in many sources. One of the important problems in particular is the drift parameter estimation. In this regard, let us mention papers \cite{HuNu}  and \cite{KlLeBr}, where the fractional Ornstein-Uhlenbeck process with unknown drift parameter originally was studied,  books \cite{bishwal}, \cite{mishura} and \cite{prakasa rao} and the references therein,  and    papers \cite{betotu}, \cite{xiaoZX}, \cite{xiaoZZ}, and \cite{HuXZ}, where the estimate was constructed via discrete observations. We shall also use the results for sequential estimates for semimartingales from \cite{melnnov}.
In the present paper we consider   stochastic differential equations involving fractional Brownian motion along with  equations involving both standard and fractional Brownian motion. We derive the standard maximum likelihood estimate and propose non-standard estimates for the unknown drift parameter. Several non-standard estimates for the drift parameter were proposed in \cite{HuNu} for the fractional Ornstein-Uhlenbeck process. We go a step ahead and propose non-standard estimates for the drift parameter in a general stochastic differential equation involving fBm. For the models involving only fractional Brownian motion, we compare  properties of the estimates. In the mixed models the standard maximum likelihood estimate does not exist but the non-standard estimate works. To formulate the conditions for strong consistency of the non-standard estimates, we need to investigate the  asymptotic behavior of the fractional derivative of the fractional Brownian motion using   the general growth results for Gaussian processes.

The paper is organized as follows. In Section 2 we introduce the models and the estimates: the maximum likelihood estimate, two non-standard estimates and three sequential estimates. Asymptotic growth of the fractional derivative of fBm is established in Section 4. Section 5 contains the main results concerning the strong consistency of all  estimates and some additional properties   of sequential estimates. The linear model and Ornstein-Uhlenbeck model  are studied in detail. We generalize the result of strong consistency of the drift parameter estimate in the Ornstein-Uhlenbeck model from \cite{KlLeBr} to the model with variable  coefficients.

\section{Model description and preliminaries}
\subsection{Model description}
Let $(\Omega, \F,\overline{\F}, P)$ be a complete probability space with filtration $\overline{\F}=\{\F_t, t\in\RR^+\}$ satisfying the standard assumptions. It is assumed that all  processes under consideration are adapted to  filtration $\overline{\F}$.
\begin{definition}\label{def2.1} Fractional Brownian motion (fBm) with Hurst index $H \in (0,1)$
is a  Gaussian process $B^{H}=\{B_{t}^{H}, t \in \RR^+\}$
on $(\Omega, \F, P)$ featuring  the properties
\begin{itemize}
  \item[] (a)\quad $B_{0}^{H}=0$;
  \item[] (b)\quad $EB_{t}^{H}=0,  t \in \RR^+$;
  \item[] (c)\quad $EB_{t}^{H}B_{s}^{H}=\frac{1}{2}( t ^{2H}+ s ^{2H}-|t-s|^{2H}), s,t
\in \RR^+$.
\end{itemize}
\end{definition}
We consider the continuous modification of $B^H$ whose existence is guaranteed by the classical Kolmogorov theorem.

To describe the statistical model, we need to introduce the pathwise integrals w.r.t. fBm. Consider
two non-random functions
 $f$ and $g$  defined on some interval $[a,b]\subset \RR^+$.
 Suppose also
that the  the following limits exist:  $f(u+):=\lim_{\delta\downarrow0}f(u+\delta)\text{
and } g(u-):=\lim_{\delta\downarrow0}g(u-\delta),\   a\leq u\leq
b$. Let $$f_{a+}(x):=(f(x)-f(a+))\textit{1}_{(a,b)}(x), \
g_{b-}(x):=(g(b-)-g(x))\textit{1}_{(a,b)}(x).$$ Suppose that
$f_{a+}\in I_{a+}^{\alpha}(L_p[a,b])), \ g_{b-}\in
I_{b-}^{1-\alpha}(L_q[a,b]))$ for some $p\geq 1, \ q\geq 1,
1/p+1/q\leq 1, \ 0\leq \alpha \leq 1.$ (For the standard notation and
statements concerning fractional analysis, see \cite{Samko}).
Introduce the fractional derivatives
\begin{equation*}\label{eq1.1}(\mathcal{D}_{a+}^{\alpha}f_{a+})(x)=\frac{1}{\Gamma(1-\alpha)}\Big(\frac{f_{a+}(s)}{(s-a)^\alpha}+\alpha
\int_{a}^s\frac{f_{a+}(s)-f_{a+}(u)}{(s-u)^{1+\alpha}}du\Big)1_{(a,b)}(x)\end{equation*}
\begin{equation*}\label{eq1.2}(\mathcal{D}_{b-}^{1-\alpha}g_{b-})(x)=\frac{e^{-\emph{i}\pi
\alpha}}{\Gamma(\alpha)}\Big(\frac{g_{b-}(s)}{(b-s)^{1-\alpha}}+(1-\alpha)
\int_{s}^b\frac{g_{b-}(s)-g_{b-}(u)}{(s-u)^{2-\alpha}}du\Big)1_{(a,b)}(x).\end{equation*}
It is known that
 $\mathcal{D}_{a+}^{\alpha}f_{a+}\in L_p[a,b], \ \mathcal{D}_{b-}^{1-\alpha}g_{b-}\in
L_q[a,b].$

\begin{definition}\label{def2.2} (\cite{Zah98}, \cite{Zah99}) Under above
assumptions, the generalized (fractional) Lebesgue-Stieltjes
integral $\int_a^bf(x)dg(x)$ is defined as
$$\int_a^bf(x)dg(x):=e^{\emph{i}\pi
\alpha}\int_a^b(\mathcal{D}_{a+}^{\alpha}f_{a+})(x)(\mathcal{D}_{b-}^{1-\alpha}g_{b-})(x)dx+f(a+)(g(b-)-g(a+)),$$
and for $\alpha p<1$ it can be simplified to
$$\int_a^bf(x)dg(x):=e^{\emph{i}\pi
\alpha}\int_a^b(\mathcal{D}_{a+}^{\alpha}f)(x)(\mathcal{D}_{b-}^{1-\alpha}g_{b-})(x)dx.$$
\end{definition}

As  follows from \cite{Samko}, for any $1-H<\alpha<1$ there
exist  fractional derivatives $\mathcal{D}_{b-}^{1-\alpha}B_{b-}^H$ and $\mathcal{D}_{b-}^{1-\alpha}B_{b-}^H\in
L_{\infty}[a,b]$ for any $0\leq a<b.$ Therefore, for $f\in I_{a+}^{\alpha}(L_1[a,b])$
we can define the integral w.r.t. fBm in the following way.

\begin{definition}\label{def2.3} (\cite{NR}, \cite{Zah98}, \cite{Zah99}) The
integral with respect to fBm is defined as
\begin{equation}\label{eq2.1} \int_a^bfdB^H:=e^{\emph{i}\pi
\alpha}\int_a^b(\mathcal{D}_{a+}^{\alpha}f)(x)(\mathcal{D}_{b-}^{1-\alpha}B_{b-}^H)(x)dx.\end{equation}
\end{definition}
An  evident estimate follows immediately from \eqref{eq2.1}:
\begin{equation}\label{eq2.2}\Big|\int_a^bfdB^H\Big|\leq\sup_{a\leq x\leq b}|(\mathcal{D}_{b-}^{1-\alpha}B_{b-}^H)(x)|\int_a^b|(\mathcal{D}_{a+}^{\alpha}f)(x)|dx.\end{equation}

Let us take a   Wiener process $W=\{W_t, t\in\RR^+\}$ on  probability space $(\Omega, \F,\overline{\F}, P)$, possibly correlated with $B^H$. Assume that $H>\frac12$ and consider a  one-dimensional mixed stochastic differential equation involving both the Wiener process and the fractional Brownian motion
 \begin{equation}\label{mixed} X_t =x_0 +\theta\int_0^t a(s,X_{s})ds+
\int_0^tb(s,X_{s})dB_s^H+\int_0^tc(s,X_{s})dW_s,\
t\in\RR^+,\end{equation} where $x_0\in\RR$ is the initial value, $\theta$ is the unknown parameter to be estimated, the first integral in the right-hand side of \eqref{mixed} is
the Lebesgue-Stieltjes integral, the second integral is the
generalized Lebesgue-Stieltjes integral introduced in  Definition \ref{def2.3}, and the third one is the It\^{o} integral.
From now on, we  shall assume that the coefficients of  equation \eqref{mixed} satisfy the following assumptions on any interval $[0,T]$:
\begin{list}{Hypotheses}{\parsep=-1mm}
\item[$(A_1)$] \emph{Linear growth of $a$ and $b$}: for any $ s\in [ 0,T]$ and any $x\in
\mathbb{R}$ \begin{equation*} | a(s,x)|  +| b(s,x)| \leq
K(1+\abs{x}).\end{equation*}
\item[$(A_2)$] \emph{Lipschitz continuity  of $a,c$ in space}: for any  $t\in[0,T]$ and $x,y\in \mathbb R$
$$|a(t,x)-a(t,y)|+|c(t,x)-c(t,y)|\le K|x-y|.$$
\item[$(A_3)$] \emph{H\"older continuity in time}:  function $b(t,x)$ is differentiable in $x$ and there exists $\beta\in(1-H,1)$
such that for any ${s,t\in[0,T]}$ and any ${x\in \mathbb R}$
$$|a(s,x)-a(t,x)|+|b(s,x)-b(t,x)|+|c(s,x)-c(t,x)|+|\partial_x{}b(s,x)-\partial_x{}b(t,x)|\le{K}|s-t|^{\beta}.$$
\item[$(A_4)$] \emph{Lipschitz continuity of $\partial_x b$ in space}: for any  $t\in[0,T]$ and any $ x,y\in \mathbb R$  $$
|\partial_x{}b(t,x)-\partial_x{}b(t,y)|\le{K}|x-y|.
$$
\item[$(A_5)$] \emph{Boundedness of $c$ and $\partial_x b$}: for any $s\in[0,T]$ and $x\in \mathbb{R}$
$$
\abs{c(s,x)}+\abs{\partial_x b(s,x)} \leq K.
$$
\end{list}
Here $K$ is  a constant independent of $x$, $y$, $s$ and $t$. For an arbitrary interval $[0,T], \;\alpha>0$  and $\kappa=\frac12\wedge\beta $ define the following norm:

$$\norm{f}_{\infty,\alpha,[0,T]} = \sup_{s\in [0,T]} \left(\abs{f(s)} + \int_0^s \abs{f(s)-f(z)}(s-z)^{-1-\alpha}dz\right).$$ It was proved in  \cite{MiSh} that under assumptions $(A_1)-(A_5)$   there exists   solution $X=\{X_t, \F_t, t\in [0,T]\}$ for equation \eqref{mixed} on any interval $[0,T]$ which satisfies \begin{equation}\label{solution-cond}
\norm{X}_{\infty,\alpha,[0,T]}<\infty\quad\text{a.s.}
\end{equation} for any $\alpha\in(1-H,\kappa)$. This solution is unique in the class of processes satisfying \eqref{solution-cond} for some $\alpha>1-H$.

\begin{remark} In  case when  components $W$ and $B^H$ are independent, assumptions for the coefficients can be relaxed, as it has been shown in \cite{GuNu}. More specifically,  coefficient  $c$ can be of linear growth, and $\partial_x{}b$ can be  H\"{o}lder continuous up to some order less than 1.
\end{remark}

\subsection{Construction of drift parameter  estimates: the standard maximum likelihood estimate.} To start with, consider the case $c(t,x)\equiv0$ which was studied, for instance, in  \cite{KlLeBr} and \cite{mishura}. Recall some facts from the theory of drift parameter estimation in this case. Consider the equation
\begin{equation}\label{main} X_t =x_0 +\theta\int_0^t a(s,X_{s})ds+
\int_0^tb(s,X_{s})dB_s^H,\
t\in\mathbb{R}.\end{equation}

Let  assumptions $(A_1)$ and   $(A_3)$ with $c\equiv0$ hold on any interval $[0,T]$, together with the following assumptions:
\begin{list}{Hypotheses}{\parsep=-1mm}

\item[$(A'_2)$] \emph{Lipschitz continuity of $a,b$ in space}: for any  $t\in[0,T]$ and $x,y\in \mathbb R$
$$|a(t,x)-a(t,y)|+|b(t,x)-b(t,y)|\le K|x-y|,$$

\item[$(A'_4)$] \emph{H\"{o}lder continuity  of $\partial_x{}b(t,x)$ in space:}    there exists such
 $\rho \in(3/2 - H,1)$ that
for any  $t\in[0,T]$ and $ x,y\in\RR$  $$
|\partial_x{}b(t,x)-\partial_x{}b(t,y)|\le{D}|x-y|^\rho,
$$

\end{list}
Then, according to \cite{NR},  solution for  equation \eqref{main} exists  on any interval $[0,T]$ and is unique
in the class of processes satisfying \eqref{solution-cond} for some  $\alpha>1-H$.

In addition, suppose that the following assumption holds:
\begin{list}{Hypotheses}{\parsep=-1mm}
\item[$(B_1)$] $b(t,X_t)\neq 0, t\in[0,T]$ and
$\frac{a(t,X_t)}{b(t,X_t)}$ is a.s. Lebesgue integrable on $[0,T]$ for any $T>0$.
\end{list}
Denote $\psi(t, x)=\frac{a(t,x)}{b(t,x)}$, $\varphi(t):=\psi(t, X_t)$. Also, let the kernel $$l_{H}(t,s)=c_{H}s^{\frac12-H}(t-s)^{\frac12-H}I_{\{0<s<t\}},$$ with $c_{H}=\left(\frac{\Gamma(3-2H)}{2H\Gamma(\frac32-H)^{3}\Gamma(H+\frac12)}\right)^
 {\frac{1}{2}}$, and introduce the integral \begin{equation}\label{frac}J_t=\int _0^t l_{H}(t,s)\varphi(s)ds=c_{H} \int _0^t (t-s)^{\frac12-H}s^{\frac12-H}\varphi(s)ds. \end{equation} Finally, let  $ M_t^H =\int_0^t l_H(t,s)dB_s^H $ be Gaussian martingale with square bracket $\langle M\rangle_t^H=t^{2-2H}$ (Molchan martingale, see \cite {nvv}).

 Consider two processes: $$Y_t=\int_0^tb^{-1}(s, X_s)dX_s=\theta\int_0^t\varphi(s)ds+B_t^H$$ and
 $$Z_t=\int _0^t l_{H}(t,s)dY_s=\theta J_t+M_t^H.$$
 Note that we can rewrite  process $Z$ as $$Z_t=\int_0^tl_{H}(t,s)b^{-1}(s, X_s)dX_s,$$ so $Z$ is a functional of the observable process $X$.  The following smoothness condition for the function $\psi$ (Lemma 6.3.2 \cite{mishura})  ensures the  semimartingale property of $Z$.
\begin{lemma}\label{lem2.1} Let $\psi(t, x)\in C^1(\RR^+)\times C^2(\RR).$
Then  for any $t>0$

\begin{equation}\label{ziprime}
\begin{gathered}J'(t)=(2-2H)C_H\psi(0,x_0)t^{1-2H}
+
\int_0^tl_H(t,s)\left(\psi_t'(s,X_s)+\theta\psi_x'(s,X_s)a(s,X_s)\right)ds\\
-\Big(H-\frac12\Big)
c_H\int_0^ts^{-\frac12-H}(t-s)^{\frac12-H}\int_0^s\Big(\psi_t'(u,X_u)+\theta\psi_x'(u,X_u)a(u,X_u)\Big)duds\\
+(2-2H)c_Ht^{1-2H}\int_0^t
s^{2H-3}\int_0^su^{\frac32-H}(s-u)^{\frac12-H}\psi_x'(u,X_u)b(u,X_u)dB_u^Hds\\
+c_Ht^{-1}\int_0^tu^{\frac32-H}(t-u)^{\frac12-H}\psi_x'(u,X_u)b(u,X_u)dB_u^H,\end{gathered}\end{equation}
where  $C_H=B(\frac32-H,\frac32-H)c_H=\Big(\frac{\Gamma(\frac32-H)}{2H\Gamma(H+\frac12)\Gamma(3-2H)}\Big)^{\frac12},$ and all of the involved integrals  exist a.s.
%
\end{lemma}
\begin{remark}\label{remzhiprime} Suppose that $\psi(t, x)\in C^1(\RR^+)\times C^2(\RR)$ and  limit $\varsigma(0)=\lim_{s\rightarrow0}\varsigma(s)$ exists a.s., where $\varsigma(s)=s^{\frac12-H}\varphi(s)$. In this case  $J(t)$ can be presented as   $$J(t)=c_H\int_0^t(t-s)^{\frac12-H}\varsigma(s)ds = \frac{c_Ht^{\frac32-H}}{\frac32-H}\varsigma(0)+c_H\int_0^t\frac{(t-s)^{\frac32-H}}{\frac32-H}\varsigma'(s)ds,$$ and $J'(t)$ from \eqref{ziprime} can be simplified to
\begin{equation*}
\begin{gathered}J'(t)=c_Ht^{\frac12-H}\varsigma(0)+\int_0^tl_H(t,s)\Big(\Big(\frac12-H\Big)s^{-1}\varphi(s)
+\psi_t'(s,X_s)\\+\theta\psi_x'(s,X_s)a(s,X_s)\Big)ds+\int_0^tl_H(t,s)\psi_x'(s,X_s)b(s,X_s)dB_s^H.
\end{gathered}
\end{equation*}
\end{remark}

Same way as $Z $,  processes $J $ and $J' $ are  functionals of $X$. It is more convenient to consider  process $\chi(t)=(2-2H)^{-1}J'(t)t^{2H-1}$, so that  $$Z_t=(2-2H)\theta \int_0^t\chi(s)s^{1-2H}ds+M_t^H=\theta \int_0^t\chi(s)d\langle M^H\rangle _s+M_t^H.$$

Suppose that the following conditions hold:

\begin{enumerate} \item [$(B_2)$] $EI_T:=E\int_0^T\chi^2_sd\langle M^H\rangle _s<\infty$ for any $T>0$,

\item [$(B_3)$] $I_\infty:=\int_0^\infty\chi^2_sd\langle M^H\rangle _s=\infty$
a.s.
\end{enumerate}
Then we can consider the maximum likelihood estimate $$\theta^{(1)}_T=\frac{\int_0^T\chi_sdZ_s}{\int_0^T\chi^2_sd\langle M^H\rangle _s}=\theta+\frac{\int_0^T\chi_sdM^H_s}{\int_0^T\chi^2_sd\langle M^H\rangle _s}.$$
 Condition $(B_2)$ ensures that  process $\int_0^t\chi_sdM^H_s, t>0$ is a square integrable martingale, and condition $(B_3)$ alongside with the law of large numbers for martingales ensure   that $\frac{\int_0^T\chi_sdM^H_s}{\int_0^T\chi^2_sd\langle M^H\rangle _s}\rightarrow 0$ a.s. as $T\rightarrow\infty$. Summarizing, we arrive at the following result (\cite{mishura}).
 \begin{proposition}\label{pro_2.1} Let $\psi(t, x)\in C^1(\RR^+)\times C^2(\RR) $ and assumptions $(A_1)$, $(A_3)$, $(A'_2)$, $(A'_4)$ and $(B_1)$--$(B_3)$ hold.  Then   estimate $\theta^{(1)}_T$ is strongly consistent as $T\rightarrow\infty$.
 \end{proposition}

\subsection{Construction of drift parameter  estimates: two non-standard estimates. } In  case when $c=0$, it is possible to construct another estimate for  parameter $\theta$,    preserving the structure of the standard maximum likelihood estimate. Similar approach was applied in \cite{HuNu} to the fractional Ornstein-Uhlenbeck process with constant coefficients. We shall use process $Y$ to  define the estimate as  \begin{equation}\label{teta2}\theta^{(2)}_T=\frac{\int_0^T\varphi_sdY_s}{\int_0^T\varphi^2_sds}
 =\theta+\frac{\int_0^T\varphi_sdB^H_s}{\int_0^T\varphi^2_sds}.\end{equation}

 Let us return to general  equation \eqref{mixed} with  non-zero $c$ and construct the  estimate of parameter $\theta$. Suppose that the following assumption holds: \begin{enumerate} \item [$(C_1)$] $c(t,X_t)\neq 0, t\in[0,T]$,
$\frac{a(t,X_t)}{c(t,X_t)}$ is a.s. Lebesgue integrable on $[0,T]$ for any $T>0$ and there exists generalized Lebesgue--Stieltjes integral  $\int_0^T \frac{b(t,X_t)}{c(t,X_t)}dB_t^H$.\end{enumerate}

 Define   functions  $\psi_1(t, x)=\frac{a(t,x)}{c(t,x)}$ and  $\psi_2(t, x)=\frac{b(t,x)}{c(t,x)}$,  processes $\varphi_i(t)=\psi_i(t, X_t), i=1,2$ and   process $$Y_t=\int_0^tb^{-1}(s, X_s)dX_s=\theta\int_0^t\varphi_1(s)ds+\int_0^t\varphi_2(s)dB_s^H+W_t.$$
 Evidently, $Y$ is a functional of $X$ and is observable. Assume additionally   that the  generalized Lebesgue--Stieltjes integral  $\int_0^T \varphi_1(t)\varphi_2(t)dB_t^H$ exists and

 \begin{enumerate} \item [$(C_2)$] for any $T>0$ $E\int_0^T\varphi_1^2(s)ds<\infty. $\end{enumerate}

   Denote $\vartheta(s)=\varphi_1(s)\varphi_2(s)$. We can consider the following estimate of parameter $\theta$:
 \begin{equation}\label{eq.teta3} {\theta}_T^{(3)}=\frac{\int_0^T\varphi_1(s)dY_s}{\int_0^T\varphi_1^2(s)ds}
 =\theta+\frac{\int_0^T\vartheta(s)dB^H_s}{\int_0^T\varphi_1^2(s)ds}+\frac{\int_0^T\varphi_1(s)dW_s}
 {\int_0^T\varphi_1^2(s)ds}.\end{equation}

 Estimate  ${\theta}_T^{(3)}$ preserves the
   traditional form of maximum likelihood estimates for diffusion models. The right-hand side of \eqref{eq.teta3} provides
    a stochastic representation of ${\theta}_T^{(3)}$.  We shall use  it to investigate the strong consistency of this estimate.
\subsection{Construction of drift parameter  estimates: sequential  estimates. } Return to model \eqref{main} and suppose that
     conditions $(B_1)-(B_3)$ hold. For any $h>0$ consider the stopping time $$\tau(h)=\inf\{t>0: \int_0^t\chi^2_sd\langle M\rangle_s=h\}.$$
     Under conditions $(B_1)-(B_2)$ we have $\tau(h)<\infty$ a.s. and $\int_0^{\tau(h)}\chi^2_sd\langle M\rangle_s=h$. The sequential maximum likelihood estimate has a form \begin{equation}\label{3.30}
     {\theta}_{\tau(h)}^{(1)}=\frac{\int_0^{\tau(h)}\chi_sdZ_s}{h}
     =\theta+\frac{\int_0^{\tau(h)}\chi_sdM^H_s}{h}.\end{equation}
Sequential versions of estimates ${\theta}_T^{(2)}$ and ${\theta}_T^{(3)}$ have a form
$${\theta}_{\tau(h)}^{(2)}=\theta+\frac{\int_0^{\tau(h)}\varphi_sdB^H_s}{h} $$ and
$${\theta}_{\upsilon(h)}^{(3)}=\theta
+\frac{\int_0^{\upsilon(h)}\vartheta(s)dB^H_s}{h}+\frac{\int_0^{\upsilon(h)}\varphi_1(s)dW_s}{h}, $$
where $$\upsilon(h)=\inf\{t>0: \int_0^t\varphi^2_1(s)ds=h\}.$$

To provide an exhaustive study of the introduced estimates, we will need a number of  auxiliary facts about Gaussian
processes. These facts are presented in the next section. Technical proofs may be found in Appendix.

 \section{Auxiliary results for Gaussian processes related to the fractional Brownian motion.}
 We start with the exponential maximal bound for a Gaussian process defined on an abstract  pseudometric space, expressed in terms of the metric capacity of this space. This result is a particular case of the general theorem proved in  \cite{BulKoz}, p. 100.
\begin{lemma}\label{lem_2}
Let $\mathbf{T}$ be a non-empty  set,
$X = \{X(\textbf{t}),\ \textbf{t} \in \mathbf{T}\}$ be centered Gaussian process. Suppose that the  pseudometric space $(\mathbf{T},\rho)$  with  pseudometric
$$
\rho\left(\textbf{t},\textbf{s}\right) = \left(E(X(\textbf{t}) - X(\textbf{s}))^2\right)^{\frac{1}{2}}
$$
is separable and  process $X$ is separable on this space. Also, let the following conditions hold:
\begin{equation*}\label{riv_l}
a:= \sup_{t \in \mathbf{T}}\left(E|X(\textbf{t})|^2\right)^\frac{1}{2} < \infty,
\end{equation*}
and
\begin{equation*}\label{riv_m}
\int_0^{a}(\log N_\textbf{T}(u))^\frac{1}{2}du < \infty,
\end{equation*}
where $N_\textbf{T}(u)$ is the number of elements in the minimal $u$-covering of  space $(\mathbf{T},\rho)$. Then for any  $\lambda > 0$ and any $\theta\in(0,1)$ the following inequality holds:
\begin{equation*}\label{riv_n}
E\exp\left\{\lambda\sup_{t \in \mathbf{T}}|X(\textbf{t})|\right\} \leq 2Q(\lambda,\theta),
\end{equation*}
where
$$
Q(\lambda,\theta) =  \exp \left\{\frac{\lambda^2a^2}{2(1 - \theta)^2} + \frac{2\lambda}{\theta(1 - \theta)}\int_0^{\theta a}\left(\log(N_\textbf{T}(u))\right)^\frac{1}{2}du \right\}.
$$
\end{lemma}

Consider   set $\mathbf{T} = \{\textbf{t}=(t_1,t_2)\in\mathbb{R}^2_+: 0 \leq t_2 \leq t_1\}$  supplied with the distance
$$
m\left(\textbf{t},\textbf{s}\right) = |t_1-s_1|\vee|t_2-s_2|.
$$
Assume  random process $X = \{X(\textbf{t}),\textbf{t}\in\mathbf{T}\}$ satisfies the following conditions.
\begin{enumerate}\item [$(D_1)$] Process $X$ is a centered Gaussian process on $\mathbf{T}$,
 separable on  metric space  $(\mathbf{T},m)$.

\item [$(D_2)$] There exist   $\beta>0, \gamma>0$ and  a constant ${C}(\beta, \gamma)$ independent of $X$, $\textbf{t}$ and $\textbf{s}$ such that for any $\textbf{t},\textbf{s}\in\textbf{T}$
\begin{equation}\label{riv_f}
\left(E(X(\textbf{t}) - X(\textbf{s}))^2\right)^{\frac{1}{2}}
\leq{C}(\beta, \gamma)\left(t_1\vee s_1 \right)^\beta\left(m\left(\textbf{t},\textbf{s}\right)\right)^\gamma.
\end{equation}
\item [$(D_3)$] There exist $\delta>0$ and  a constant ${C}(\delta)$ independent of $X$ and $\textbf{t}$ such that for any $\textbf{t}\in\textbf{T}$
\begin{equation}\label{riv_e}
\left(E(X(\textbf{t}))^2\right)^{\frac{1}{2}} \leq {C}(\delta)
t_1^\delta.
\end{equation}
\end{enumerate}

Let us introduce the following notations. Let $A(t) > 1, t \geq 0$ be an increasing function such that $A(t) \rightarrow \infty$, $t \rightarrow \infty$. Consider an increasing sequence
  $b_0 = 0$, $b_\ell < b_{\ell+1},l\geq1 $ and suppose that $ b_\ell \rightarrow \infty ,\; \ell\rightarrow\infty$. For
$\delta_\ell = A(b_\ell)$ and $\kappa>0$ we denote  $$S(\delta)=\sum_{\ell = 0}^{\infty}b_{\ell+1}^{\delta }\delta_\ell^{-1},\;
\kappa_1 = \frac{\kappa}{2}\left(1 + \frac{\beta}{\gamma} - \frac{\delta}{\gamma}\right),\;B_1 = {C}(\delta)S(\delta),$$
$$C_1 = C_2\kappa^{-\frac{1}{2}}S(\delta+\kappa_1)\;\text{and}\; C_2= \frac{2^{\frac{1-\kappa}{2}}}{{1-\frac{\kappa}{2\gamma}}}({{C}(\delta)})^{1-\frac{\kappa}{2\gamma}}
({C}(\beta, \gamma))^{\frac{\kappa}{2\gamma}}.$$ Now we shall present the  auxiliary exponential maximal bound for a Gaussian process defined on $(\mathbf{T},m)$.

\begin{theorem}\label{teo_2}
Let $\{X(\textbf{t}), \textbf{t}\in \textbf{T}\}$ be a random process satisfying assumptions $(D_1)-(D_3)$. Let  $0\leq a <b$, set  $\textbf{T}_{a,b} = \{\textbf{t}=(t_1,t_2)\in \textbf{T} : a \leq t_1 \leq b,\; 0\leq t_2\leq t_1\}$. Then for any
 $
0 < \theta < 1,\ \lambda > 0\quad\text{and}\quad 0 < \kappa < 1\wedge2\gamma
$
the following inequality holds:
\begin{equation*}\label{riv_o}
E\exp\left\{\lambda\sup_{\textbf{t} \in \textbf{T}_{a,b}}|X(\textbf{t})|\right\} \leq 2\widetilde{Q}(\lambda,\theta),
\end{equation*}
where
$$
\widetilde{Q}(\lambda,\theta) = \exp \left\{\frac{\lambda^2(b^\delta {C}({\delta}))^2}{2(1 - \theta)^2}
 + \frac{2\lambda}{1 - \theta}b^{\delta + \kappa_1}\frac{C_2}{\theta^{\frac{\kappa}{2\gamma}}\kappa^{\frac12}}\right\}.
$$
\end{theorem}
\begin{proof}
It follows from  (\ref{riv_e}) and  (\ref{riv_f}) that
\begin{equation}\label{riv_p}
d := \sup_{\textbf{t} \in \textbf{T}_{a,b}}\left(E|X(\textbf{t})|^2\right)^\frac{1}{2}\leq {C}(\delta)b^\delta
\end{equation}
and
\begin{equation}\label{riv_q}
\sup_{m(\textbf{t},\textbf{s}) \leq h, \textbf{t}, \textbf{s }\in \textbf{T}_{a,b}}\left(E(X(\textbf{t}) - X(\textbf{s}))^2\right)^\frac{1}{2} \leq \sigma(h) :={C}(\beta,\gamma) b^\beta h^\gamma.
\end{equation}
In turn, it follows from (\ref{riv_q}) that
\begin{equation}\label{riv_r}
N_{\textbf{T}_{a,b}}(v) \leq \left(\frac{b - a}{2\sigma^{(-1)}(v)} + 1\right)\left(\frac{b}{2\sigma^{(-1)}(v)} + 1\right) \leq \left(\frac{({C}(\beta,\gamma))^\frac{1}{\gamma}b^{1+\frac{\beta}{\gamma}}}{2v^\frac{1}{\gamma}} + 1\right)^2.
\end{equation}
Define $J\left(\theta d\right):= \int_0^{\theta d}\left(\log {N_{\textbf{T}_{a,b}}(u)}\right)^\frac{1}{2}d{u}.$ It follows from  (\ref{riv_r}) that
\begin{equation}\label{riv_s}
J\left(\theta d\right) \leq \int_0^{\theta d}\sqrt{2}\left[\log\left(\frac{({C}(\beta,\gamma))^\frac{1}{\gamma}b^{1+\frac{\beta}{\gamma}}}{2v^\frac{1}{\gamma}} + 1 \right)\right]^\frac{1}{2}dv.
\end{equation}
For any $0 < \kappa \leq 1$,
$$
\log(1 + x) =  \frac{1}{\kappa}\log(1 + x)^{\kappa} \leq \frac{x^{\kappa}}{\kappa}.
$$
Now, let $\kappa\in(0, 1\wedge2\gamma)$. Then it follows from \eqref{riv_p} and \eqref{riv_s} that
\begin{multline*}\label{riv_u}
J\left(\theta d\right) \leq  \frac{\sqrt{2}}{\kappa^{\frac{1}{2}}}\int_0^{\theta d}\frac{(({C}(\beta,\gamma))^\frac{1}{\gamma}b^{1+\frac{\beta}{\gamma}})^\frac{\kappa}{2}}
{(2v^\frac{1}{\gamma})^\frac{\kappa}{2}}dv \\
= \frac{\sqrt{2}}{{\kappa^\frac{1}{2}}(1-\frac{\kappa}{2\gamma})}
\left(\frac{({C}(\beta,\gamma))^\frac{1}{\gamma}b^{1+\frac{\beta}{\gamma}}}{2}\right)^\frac{\kappa}{2}(\theta d)^{1 - \frac{\kappa}{2\gamma}} \leq b^{\delta + \kappa_1}\frac{\theta^{1-\frac{\kappa}{2\gamma}}}{\kappa^\frac{1}{2}}C_2.
\end{multline*}
Separability of $X$ on  $(\mathbf{T},m)$ and  relation  (\ref{riv_q}) ensure separability of  $X$ on $(\mathbf{T},\rho)$ with $
\rho\left(\textbf{t},\textbf{s}\right) = \left(E(X(\textbf{t}) - X(\textbf{s}))^2\right)^{\frac{1}{2}}
$. Hence the statement of the theorem follows from Lemma  \ref{lem_2}.
\end{proof}

Now we are ready to state the general  result concerning the asymptotic maximal growth of a Gaussian process defined on $(\mathbf{T},m)$.
\begin{theorem}\label{teo2}
Let $X = \{X(\textbf{t}),\textbf{t}\in\mathbf{T} \}$  satisfy   assumptions $(D_1)-(D_3)$.
Suppose that  function $A(t)$ is chosen in such a way that  series $S(\delta)$ converges.
In  case when $1 + \frac{\beta}{\gamma} - \frac{\delta}{\gamma}>0$, assume additionally that there exists such $0 < {\kappa} < 1$ that
 series $S(\delta+\kappa_1)$ converges with
$
\kappa_1 = \frac{\kappa}{2}\left(1 + \frac{\beta}{\gamma} - \frac{\delta}{\gamma}\right).$

Then there exists such random variable  $\xi > 0$ that on any $\omega \in\Omega$ and for any $\textbf{t}\in\mathbf{T}$ $$|X(\textbf{t})| \leq A(t_1)\xi,$$
and $\xi$ satisfies the following assumption:

$(D_4)$  for any  $\varepsilon > (2C_1 + 1)^{\frac{2\gamma} {2\gamma+\kappa}}$
$$
P\{\xi > \varepsilon\} \leq 2\exp\left\{-\frac{\left(\varepsilon - \varepsilon^\frac{{\kappa}}{{2\gamma + \kappa}}(2C_1+1 )\right)^2}{2B_1^2}\right\}.
$$
Here the value of $\kappa<2\gamma$ is chosen to ensure the convergence of  series $S(\delta+\kappa_1)$  in  case when $1 + \frac{\beta}{\gamma} - \frac{\delta}{\gamma}>0$, and we set $\kappa
=\frac{1}{2}\wedge\gamma$ in  case when $1 + \frac{\beta}{\gamma} - \frac{\delta}{\gamma}\leq 0$.

\end{theorem}

\begin{proof}
It is easy to check that
\begin{equation}\label{riv_v}
I := E\exp\left\{\lambda\sup_{\textbf{t} \in \mathbf{T}}\frac{|X(\textbf{t})|}{A(t_1)}\right\} \leq E\exp\left\{\lambda\sum_{\ell = 0}^{\infty}(\delta_\ell)^{-1}\sup_{t_1 \in (b_\ell,b_{\ell+1}) }|X(\textbf{t})|\right\}.
\end{equation}
Let  $\ell \geq 0,\,r_\ell > 1 $ be such integers that $\sum_{\ell = 0}^\infty \frac{1}{r_\ell} = 1$. Then it follows from (\ref{riv_v}), Theorem  \ref{teo_2} and H\"{o}lder inequality that for any $\theta\in(0,1)$ and $
 0 < \kappa < 1\wedge2\gamma
$
$$
I \leq \prod_{\ell = 0}^\infty\left(E\exp\left\{\lambda\frac{r_\ell}{\delta_\ell}\sup_{t_1 \in (b_\ell,b_{\ell+1}) }|X(\textbf{t})|\right\}\right)^\frac{1}{r_\ell} \leq \prod_{\ell = 0}^\infty(2Q_\ell(\lambda,\theta))^\frac{1}{r_\ell} = 2\prod_{\ell = 0}^\infty(Q_\ell(\lambda,\theta))^\frac{1}{r_\ell},
$$
where
\begin{equation*}\label{riv_w}
Q_\ell(\lambda,\theta) = \exp\left\{\frac{\lambda^2r_\ell^2}{2\delta_\ell^2}\frac{(b_\ell^{\delta}{C}(\delta))^2}{(1 - \theta)^2} + \frac{2\lambda r_\ell}{(1 - \theta)\delta_\ell}b_\ell^{\delta +\kappa_1}\frac{C_2}{\theta^\frac{\kappa}{2\gamma} \kappa^\frac{1}{2}}\right\}.
\end{equation*}

Therefore, if we take such value of $\kappa<2\gamma$ that  series $S(\delta+\kappa_1)$ converges in  case when $1 + \frac{\beta}{\gamma} - \frac{\delta}{\gamma}>0$ and set $\kappa=\frac{1}{2}\wedge\gamma$ in  case when $1 + \frac{\beta}{\gamma} - \frac{\delta}{\gamma}\leq 0$, we obtain
\begin{equation}\label{eq3.34}
I \leq 2\exp\left\{\frac{\lambda^2({C}(\delta))^2}{2(1 - \theta)^2}\sum_{\ell = 0}^\infty\frac{r_\ell(b_\ell^\delta)^2}{\delta_\ell^2} + \frac{2\lambda C_2\kappa^{-\frac12}S(\delta+\kappa_1)}{(1 - \theta)\theta^\frac{\kappa}{2\gamma}}\right\}
\end{equation}
Now we can substitute $r_\ell = S(\delta)b_\ell^{-\delta} \delta_\ell$
into \eqref{eq3.34}:
\begin{equation*}\label{riv_z}
I \leq 2\exp\left\{\frac{\lambda^2(S(\delta){C}(\delta))^2}{2(1 - \theta)^2}+ \frac{2\lambda C_2\kappa^{-\frac12}S(\delta+\kappa_1)}{(1 - \theta)\theta^\frac{\kappa}{2\gamma}}\right\}.
\end{equation*}
Therefore,
\begin{equation}\label{riv_aa}
E\exp\left\{\lambda\sup_{\textbf{t} \in \mathbf{T}}\frac{|X(\textbf{t})|}{A(t_1)}\right\} \leq 2\exp\left\{\frac{\lambda^2}{2}\hat{B}^2 + 2\lambda \hat{C}\right\},
\end{equation}
where
$$
\hat{B}  = \frac{S(\delta){C}(\delta)}{1 - \theta} \quad\text{and}\quad
\hat{C} = \frac{C_2\kappa^{-\frac12}S(\delta+\kappa_1)}{(1 - \theta)\theta^\frac{\kappa}{2\gamma}}.
$$
It follows immediately from \eqref{riv_aa} that for any $\lambda > 0, \ \varepsilon >0$
\begin{multline}\label{eq3.35}
P\left\{\sup_{\textbf{t} \in \mathbf{T}}\frac{|X(\textbf{t})|}{A(t_1)} > \varepsilon\right\} \leq {\exp\{-\lambda\varepsilon\}}{E\exp\left\{\lambda \sup_{\textbf{t} \in \mathbf{T}}\limits\frac{|X(\textbf{t})|}{A(t_1)}\right\}} \leq\\
\leq 2\exp\left\{\frac{\lambda^2}{2}\hat{B}^2 + 2\lambda \hat{C} - \lambda\varepsilon\right\}.
\end{multline}
If we minimize the right-hand side of \eqref{eq3.35} w.r.t.  $\lambda$ then we obtain that for any  $\varepsilon > 2\hat{C}$
\begin{multline}\label{riv_ab}
P\left\{\sup_{\textbf{t} \in \mathbf{T}}\frac{|X(\textbf{t})|}{A(t_1)} > \varepsilon\right\} \leq 2\exp\left\{-\frac{(\varepsilon-2\hat{C})^2}{2\hat{B}^2}\right\}\\
 = 2\exp\left\{-\frac{(\varepsilon(1 -\theta) - 2\theta^{-\frac{\kappa}{2\gamma}}C_1)^2}{2B_1^2}\right\}.
\end{multline}
Finally, we can insert   $\theta = \varepsilon^{-\frac{2\gamma}{2\gamma+\kappa}}$ into \eqref{riv_ab} and  derive that for  $\varepsilon > (2C_1 + 1)^\frac{2\gamma}{2\gamma+\kappa}$
$$
P\left\{\sup_{t \in \mathbf{T}}\frac{|X(\mathbf{t})|}{A(t_1)} > \varepsilon\right\} \leq 2\exp\left\{-\frac{(\varepsilon - \varepsilon^\frac{\kappa}{\kappa + 2\gamma}(1 + 2C_1))^2}{2B_1^2}\right\}.
$$
Denote $\xi := \sup_{\textbf{t} \in \mathbf{T}}\limits\frac{|X(\textbf{t})|}{A(t_1)}$. Then $\xi$ satisfies  assumption $(D_4)$, and on any $\omega\in\Omega$
$$
X(\textbf{t}) \leq  A(t_1)\xi,
$$
which concludes the proof.
\end{proof}
\begin{theorem}\label{teo4}
Let $ 0 <{H}<{1}, 1-H<\alpha<1,\;\mathbf{T}=\{\mathbf{t}=(t_1, t_2),0 \leq{{t_2}}<{{t_1}}\}$,
$$
X(\textbf{t}) = \frac{B_{t_1}^H - B_{t_2}^H}{({t_1} - {t_2})^{1 - \alpha}} + \int_{{t_2}}^{{t_1}}\frac{B_u^H - B_{t_2}^H}{(u - {t_2})^{2 - \alpha}}du.
$$

Then for any $p>1$ there exists  random variable $\xi=\xi(p)$ such that for any $\mathbf{t}\in \mathbf{T}$

$$|X(\textbf{t})|\leq ((t_1^{H+\alpha-1}(\log(t_1))^{p})\vee1)\xi(p),$$ where $\xi(p) $ satisfies  assumption $(D_4)$ with some constants $B_1$ and $C_1$.
\end{theorem}
The proof of Theorem \ref{teo4} is of a technical nature and therefore it is placed in Appendix.
 \section{Main results}
 \subsection{General results on strong consistency} In this section we shall establish  conditions for strong consistency of $\theta^{(2)}_T$ and $\theta^{(3)}_T$.
 \begin{theorem}\label{thma3.1} Let   assumptions $(A_1)$, $(A_3)$, $(A'_2)$, $(A'_4)$ $(B_1)$ and $(B_2)$  hold and let   function $\varphi$ satisfy the following assumption:
 \begin{enumerate} \item [$(B_4)$] There exists such $\alpha>1-H$ and $p>1$ that
  \begin{eqnarray}\label{exa2.2}\frac{T^{H+\alpha-1}(\log T)^p\int_0^T|(\mathcal{D}_{0+}^{\alpha}\varphi)(s)|ds}{\int_0^T \varphi^2_sds}\rightarrow0\quad\text{a.s. as}\quad  T\rightarrow\infty.
     \end{eqnarray}
     \end{enumerate}

 Then  estimate $\theta^{(2)}_T$ is   correctly defined and   strongly consistent as $T\rightarrow\infty$.
 \end{theorem}
 \begin{proof} We must prove that $\frac{\int_0^T\varphi_sdB^H_s}{\int_0^T\varphi^2_sds}\rightarrow0$ a.s. as $T\rightarrow\infty.$ According to \eqref{eq2.2}, $$\Big|\int_0^T\varphi_s dB^H_s\Big|\leq\sup_{0\leq t\leq T}|(\mathcal{D}_{T-}^{1-\alpha}B_{T-}^H)(t)|\int_0^T|(\mathcal{D}_{0+}^{\alpha}\varphi)(s)|ds.$$
 Furthermore, according to  Theorem \ref{teo4}, for any $p>1$ there exists a random variable $\xi=\xi(p)$ independent of $T$   such that for any $T>0$ $$\sup_{0\leq t\leq T}|(\mathcal{D}_{T-}^{1-\alpha}B_{T-}^H)(t)|\leq \xi(p) T^{H+\alpha-1}(\log T)^p,$$
 which concludes the proof.
 \end{proof}
 Relation \eqref{exa2.2} ensures  convergence  $\frac{\int_0^T\varphi_sdB^H_s}{\int_0^T\varphi^2_sds}\rightarrow 0$ a.s. in the general case.
  In a particular case when  function $\varphi$ is non-random and  integral $\int_0^T\varphi_sdB^H_s$
   is a Wiener integral w.r.t. the fractional Brownian motion,  conditions for existence of this integral
    are simpler since assumption \eqref{exa2.2}  can be  simplified.

  \begin{theorem} \label{theorem33.2} Let   assumptions $(A_1)$, $(A_3)$, $(A'_2)$, $(A'_4)$ $(B_1)$ and $(B_2)$  hold and let function $\varphi$ be non-random and satisfy the following assumption:
  \begin{enumerate} \item [$(B_5)$]
 There exists  such $p>0$ that
 \begin{equation*}\lim \sup_{T\rightarrow\infty} \frac{T^{2H-1+p}}{\int_0^T\varphi^2(t)dt}<\infty.
 \end{equation*}
\end{enumerate}
 Then  estimate $\theta^{(2)}_T$ is      strongly consistent as $T\rightarrow\infty$.
   \end{theorem}
 \begin{proof}It follows from \cite{mmv} and the H\"{o}lder inequality that for any $r>0$ $$E\Big|\int_0^T \varphi(s)dB_s^H\Big|^r\leq C(H,r)||\varphi||^r_{L_{\frac1H}[0,T]}\leq C(H,r)||\varphi||^{r}_{L_{2}[0,T]}T^{(H-\frac12)r}.$$
 Denote $F_T=\frac{|\int_0^T\varphi(t)dB_t^H|}{\int_0^T\varphi^2(t)dt}$. Also, for any $N>1$ and any $\varepsilon>0$ define  event $A_N=\Big\{F_N>\varepsilon\Big\}$. Then \begin{equation*}P(A_N)\leq \varepsilon^{-r}\frac{E|\int_0^N \varphi(s)dB_s^H|^r}{(\int_0^N\varphi^2(t)dt)^r}\leq C(H,r)\frac{||\varphi||^r_{L_{\frac1H}[0,N]}}{||\varphi||^{2r}_{L_{2}[0,N]}}\leq C(H,r)\frac{N^{(H-\frac12)r}}{||\varphi||^{r}_{L_{2}[0,N]}}.\end{equation*} Under condition $(B_5)$ we have $P(A_N)\leq C(H,r,p)N^{-\frac{rp}{2}}$. If $r>\frac{2}{p}$, then it follows immediately from the Borel-Cantelli  lemma that  series $\sum P(A_N)$ converges, whence $F_N\rightarrow 0$ a.s. as $N\rightarrow\infty$. Now estimate the residual $$R_N=\sup_{T\in[N,N+1]}\Big|F_T-
F_N|.$$
Evidently,
$$R_N\leq\sup_{T\in[N,N+1]}\Big|\frac{\int_N^T\varphi(t)dB_t^H}{\int_0^T\varphi^2(t)dt}\Big|+F_N,$$
and it is sufficient to estimate $$R_N^1=\sup_{T\in[N,N+1]}\Big|\frac{\int_N^T\varphi(t)dB_t^H}{\int_0^T\varphi^2(t)dt}\Big|\leq \frac{\sup_{T\in[N,N+1]}\Big|\int_N^T\varphi(t)dB_t^H\Big|}{\int_0^N\varphi^2(t)dt}:=R_N^2.$$ According to Theorem 1.10.3 from \cite{mishura} and the H\"{o}lder inequality, $$E\Big(\sup_{T\in[N,N+1]}\Big|\int_N^T\varphi(t)dB_t^H\Big|\Big)^r\leq C(H,r)||\varphi||^r_{L_{\frac1H}[N,N+1]}\leq C(H,r)||\varphi||^r_{L_{2}[N,N+1]}.$$ Now we can use condition $(B_5)$ to conclude  that for any $\varepsilon>0$ \begin{equation*}\begin{gathered}P(R_N^2>\varepsilon)\leq C(H,r)\varepsilon^{-r}\frac{||\varphi||^r_{L_{2}[N,N+1]}}{||\varphi||^{2r}_{L_{2}[0,N]}}\\\leq C(H,r)\varepsilon^{-r}{||\varphi||^{-r}_{L_{2}[0,N]}}\leq C(H,r)\varepsilon^{-r}N^{-r(2H-1+p)}.\end{gathered}\end{equation*}
We can set $r>\frac{1}{2H-1+p}$ and apply  the Borel-Cantelli lemma again. Then we obtain that $R_N^2\rightarrow 0$ a.s. as $N\rightarrow 0$, which means that $\theta_T^{(2)}$ is strongly consistent. \end{proof}
 \begin{theorem} \label{theorem2.1} Let  assumptions $(C_1)$ and $(C_2)$ hold, and, in addition,
 \begin{enumerate} \item [$(C_3)$] $\int_0^T\varphi_1^2(s)ds=\infty$ a.s.
 \item [$(C_4)$] There exist  such $\alpha>1-H$ and $p>1$ that
  \begin{eqnarray}\label{exa2.3}\frac{T^{H+\alpha-1}(\log T)^p\int_0^T|(\mathcal{D}_{0+}^{\alpha}\vartheta)(s)|ds}{\int_0^T \varphi^2_{1}(s)ds}\rightarrow0\quad\text{a.s. as}\quad  T\rightarrow\infty.
     \end{eqnarray}
 \end{enumerate}

 Then  estimate ${\theta}_T^{(3)}$ is strongly consistent as $T\rightarrow\infty$.
\end{theorem}
\begin{proof} The last term in the right-hand side of \eqref{eq.teta3} tends to zero under condition $(C_3)$. The proof of convergence of the second term repeats the  proof of Theorem \ref{thma3.1}.
\end{proof}

 Similarly to Theorem \ref{theorem33.2}, conditions stated in Theorem \ref{theorem2.1}
  can be simplified in  case when  function $\vartheta$ is non-random.
\begin{theorem}\label{mixednonrandom} Let  assumptions $(C_1)$ and $(C_2)$ hold. Then, if functions $\varphi_1$ and $\varphi_2$ are non-random, function  $\varphi_1$ satisfies condition $(B_5)$, function $\varphi_2$ is bounded, then  estimate ${\theta}_T^{(3)}$ is strongly consistent as $T\rightarrow\infty$.
\end{theorem}

Now we shall take a look at the properties of sequential estimates.
 \begin{theorem} \label{sequen} \begin{enumerate} \item [$(a)$] Let assumptions $(B_1)-(B_3)$ hold. Then  estimate ${\theta}_{\tau(h)}^{(1)}$ is unbiased, efficient, strongly consistent,  $E({\theta}_{\tau(h)}^{(1)}-\theta)^2=\frac1h$, and for any estimate of the form $${\theta}_{\tau}=\frac{\int_0^{\tau}\chi_sdZ_s}{\int_0^{\tau}\chi^2_sd\langle M^H\rangle _s}
     =\theta+\frac{\int_0^{\tau}\chi_sdM^H_s}{\int_0^{\tau}\chi^2_sd\langle M^H\rangle _s}$$ with $\tau<\infty$
a.s. and $E\int_0^{\tau}\chi^2_sd\langle M^H\rangle _s\leq h$ we have that $$E({\theta}_{\tau(h)}^{(1)}-\theta)^2\leq E({\theta}_{\tau}-\theta)^2.$$
\item [$(b)$] Let  function $\varphi$ be  separated from zero, $|\varphi(s)|\geq c>0$ a.s. and satisfy the assumption: for some $1-H<\alpha<1$ and $p>0$
 \begin{equation}\label{oh-oh}\frac{\int_0^{\tau(h)}|(\mathcal{D}_{0+}^{\alpha}\varphi)(s)|ds}
 {(\tau(h))^{2-\alpha-H-p}}\rightarrow0\quad\text{a.s.}\end{equation}  as $h\rightarrow\infty$. Then  estimate ${\theta}_{\tau(h)}^{(2)}$ is strongly consistent.
    \item [$(c)$] Let  function $\varphi_1$ be  separated from zero, $|\varphi(s)|\geq c>0$ a.s. and let function  $\vartheta$
     satisfy the assumption: for some $1-H<\alpha<1$ and $p>0$ \begin{equation}\label{oh-oh-oh}\frac{\int_0^{\upsilon(h)}|(\mathcal{D}_{0+}^{\alpha}\vartheta)(s)
     |ds}{(\upsilon(h))^{2-\alpha-H-p}}\rightarrow0\quad\text{a.s.}\end{equation} as $h\rightarrow\infty$. Then  estimate ${\theta}_{\upsilon(h)}^{(3)}$ is strongly consistent.

     \item [$(d)$] Let   function $\vartheta$ be non-random, bounded and positive, $\varphi_1$ be separated from zero. Then estimate ${\theta}_{\upsilon(h)}^{(3)}$ is consistent in the following sense: for any $p>0$, $E\Big|\theta-{\theta}_{\upsilon(h)}^{(3)}\Big|^p\rightarrow 0$ as $h\rightarrow\infty$.
\end{enumerate}
\end{theorem}
\begin{proof} (a) Process $\int_0^{\tau(h)}\chi_sdM^H_s$ is a square-integrable martingale which implies that   estimate ${\theta}_{\tau(h)}^{(1)}$ is unbiased. Besides, the results from \cite{lipshir}, Chapter 17, can be applied to \eqref{3.30} directly, therefore estimate  ${\theta}_{\tau(h)}^{(1)}$ is efficient, $E({\theta}_{\tau(h)}^{(1)}-\theta)^2=\frac1h$, and for any estimate  of the form ${\theta}_{\tau}=\frac{\int_0^{\tau}\chi_sdZ_s}{\int_0^{\tau}\chi^2_sd\langle M^H\rangle _s}
     =\theta+\frac{\int_0^{\tau}\chi_sdM^H_s}{\int_0^{\tau}\chi^2_sd\langle M^H\rangle _s}$ with $\tau<\infty$
a.s. and $E\int_0^{\tau}\chi^2_sd\langle M^H\rangle _s\leq h$ we have that $E({\theta}_{\tau(h)}^{(1)}-\theta)^2\leq E({\theta}_{\tau}-\theta)^2$. Strong consistency is also evident.

(b) We have that $|\int_0^{\tau(h)}\varphi(s)dB_s^H|\leq (\tau(h))^{H+\alpha-1+p}\int_0^{\tau(h)}|(\mathcal{D}_{0+}^{\alpha}\varphi)(s)|ds$. It is sufficient to note that $h=\int_0^{\tau(h)}\varphi^2_s ds\geq c^2\tau(h)$.
The proof of statement (c) is now evident.

(d) It was proved in \cite{mishura} that in   case of non-random bounded positive function $0\leq \vartheta(s)\leq\vartheta^*$, for any stopping time $\upsilon$  $$\Big(E\Big(\sup_{0\leq t\leq\upsilon}\Big|\int_0^t \vartheta(s)dB_s^H\Big|\Big)^p\Big)^\frac1p\leq C(H,p)\vartheta^*\Big(E\upsilon^{pH}\Big)^{\frac1p}.$$
Furthermore,  same as before,   $|\varphi_1(s)|\geq c$ and $h=\int_0^{\upsilon(h)}\varphi^2_1(s) ds\geq c^2\upsilon(h)$.
These inequalities together with the Burkholder-Gundy inequality yield $$E\Big|\theta-{\theta}_{\upsilon(h)}^{(3)}\Big|^p\leq C(H,p)\Big(\frac{\vartheta^*}{c^2}h^{H-1}+h^{-\frac p2}\Big)\rightarrow0\quad\text {as} \quad h\rightarrow\infty.$$

\end{proof}
\begin{remark} Another proof of statement (a) is contained in \cite{prakasa rao}. Assumptions\eqref{oh-oh} and \eqref{oh-oh-oh} hold, for example, for bounded and Lipschitz functions $\varphi$ and $\vartheta$  correspondingly.
\end{remark}
 \subsection{Linear models and strong consistency.} I. Consider the  linear version of   model \eqref{main}: $$dX_t=\theta a(t)X_tdt+b(t) X_tdB_t^H, $$ where $a$ and $b$ are locally bounded non-random measurable functions.  In this case  solution $X$ exists, is unique and can  be presented in the integral form $$X_t=x_0+\theta\int_0^t a(s)X_sds+\int_0^tb(s) X_sdB_s^H=x_0\exp\Big\{\theta \int_0^t a(s) ds+\int_0^tb(s) dB_s^H\Big\}.$$
Suppose that function $b$ is non-zero and note that in this model $$\varphi(t)=\frac {a(t)}{b(t)}.$$ Suppose that $\varphi(t)$ is also locally bounded and
consider maximum likelihood estimate  $\theta_T^{(1)}$.    According to \eqref{frac}, to guarantee existence of  process $J'$, we have to  assume that the fractional derivative of  order $\frac32-H$ for  function  $\varsigma(s):=\varphi(s)s^{\frac12-H}$ exists and is integrable. The sufficient conditions for the existence of fractional derivatives can be found in \cite{Samko}. One of these conditions states:
\begin{enumerate} \item [$(B_6)$]
 Functions  $\varphi$ and $\varsigma$ are differentiable and their derivatives are  locally integrable.
 \end{enumerate}
  So,  the maximum likelihood estimate does not exist for an arbitrary locally bounded function $\varphi$.   Suppose  that  condition $(B_6)$ holds  and  limit $\varsigma_0=\lim_{s\rightarrow 0}\varsigma(s)$ exists. In this case, according to Lemma \ref{lem2.1} and Remark \ref{remzhiprime}, process $J'$ admits both of the following representations:
  \begin{equation*} \begin{gathered}\label{nemnogo} J'(t)=(2-2H)C_H\varphi(0) t^{1-2H}+\int_0^tl_H(t,s)\varphi'(s)ds\\
-\Big(H-\frac12\Big)c_H\int_0^ts^{-\frac12-H}(t-s)^{\frac12-H}\int_0^s\varphi'(u)duds
\\=c_H\varsigma_0t^{\frac12-H}+c_H\int_0^t(t-s)^{\frac12-H}\varsigma'(s)ds,
\end{gathered}\end{equation*}
and assuming  $(B_3)$ also holds true, the estimate $\theta^{(1)}_T$ is strongly consistent. Let us formulate some simple conditions sufficient for the strong consistency. The proof is obvious and therefore is omitted.
\begin{lemma} If function $\varphi$ is non-random, locally bounded, satisfies $(B_6)$, limit $\varsigma(0)$ exists  and one of the following assumptions hold:
\begin{enumerate} \item [(a)]
 function $\varphi$ is not identically zero and $\varphi'$ is non-negative and non-decreasing;
 \item [(b)] derivative $\varsigma'$ preserves the sign and is separated from zero;
 \item [(c)] derivative $\varsigma'$ is non-decreasing and has a non-zero limit,

 \end{enumerate}
then the estimate ${\theta}_T^{(1)}$ is strongly consistent as $T\rightarrow\infty$.
\end{lemma}
\begin{example}: Let the coefficients are constant, $a(s)=a\neq0$ and $b(s)=b\neq0$, then  the estimate has a form $\theta^{(1)}_T=\theta+\frac{b M^H_T}{aC_HT^{2-2H}}$ and is strongly consistent. In this case assumption (a) holds. In addition, power functions $\varphi(s)=s^\rho$ are appropriate for $\rho>H-1$: this can be verified directly from \eqref{frac}.
\end{example}

 Let us now apply    estimate $\theta^{(2)}_T$ to the same model. It has a form \eqref{teta2}. We can use  Theorem \ref{theorem33.2} directly  and under assumption   $(B_5)$  estimate $\theta^{(2)}_T$ is strongly consistent. Note that we do not need any assumptions on the smoothness of $\varphi$, which is a clear advantage of $\theta^{(2)}_T$. We shall consider two more examples.

 \begin{example}: If the coefficients are constant, $a(s)=a\neq0$ and $b(s)=b\neq0$, then  the estimate has a form $\theta^{(2)}_T=\theta+\frac{b B^H_T}{aT}$. We can refer to Theorem \ref{theorem33.2} and conclude that $\theta^{(2)}_T$ is strongly consistent. Alternatively, we can use Remark \ref{prosto} which states that $|B_T^H|\leq \xi T^H(\log T)^p$ for any $p>1$ and some random variable $\xi$, therefore $\frac{ B^H_T}{T}\rightarrow 0$ a.s. as $T\rightarrow \infty$. In this  case both  estimates $\theta^{(2)}_T$ and  $\theta^{(2)}_T$ are strongly consistent  and $E(\theta-\theta^{(1)}_T)^2=\frac{\gamma^2 T^{2H-2}}{a^2C_H^2}$ has the same asymptotic behavior  as $E(\theta-\theta^{(2)}_T)^2=\frac{\gamma^2 T^{2H-2}}{a^2}$.
  \end{example}
  \begin{example}: If non-random functions $\varphi$ and $\varsigma$ are bounded on some fixed interval $[0,t_0]$ but $\varsigma$ is sufficiently irregular on this interval and has no fractional derivative of order $\frac32-H$ or higher then  we can not even calculate $J'(t)$ on this interval and the maximum likelihood estimate does not exist. However, if we assume that $\varphi(t)\sim t^{H-1+\rho}$ at infinity with some $\rho>0$, then assumption $(B_5)$ holds and estimate ${\theta}_T^{(2)}$ is strongly consistent as $T\rightarrow\infty$. In this sense estimate ${\theta}_T^{(2)}$ is more flexible.
  \end{example}

 II. Consider a mixed linear model of the form \begin{equation}\label{mixedlinear}dX_t=X_t(\theta a(t)dt+b(t)dB_t^H+c(t)dW_t), \end{equation} where $a$, $b$ and $c$ are non-random measurable functions. Assume that they are locally bounded. In this case solution $X$ for equation \eqref{mixedlinear} exists, is unique and can  be presented in the integral form $$X_t=x_0\exp\Big\{\theta \int_0^t a(s) ds+\int_0^tb(s) dB_s^H+\int_0^tc(s) dW_s-\frac12\int_0^tc^2(s)ds\Big\}.$$
In what follows assume that $c(s)\neq 0$. We have that $\varphi_1(t)=\frac{a(t)}{c(t)}$ and $\varphi_2(t)=\frac{b(t)}{c(t)}$. Estimate ${\theta}_T^{(3)}$ has a form
 \begin{equation}\label{eq3.31} {\theta}_T^{(3)}=\frac{\int_0^T\varphi_1(s)dY_s}{\int_0^T\varphi_1^2(s)ds}
 =\theta+\frac{\int_0^T\varphi_1(s)\varphi_2(s)dB^H_s}{\int_0^T\varphi_1^2(s)ds}+\frac{\int_0^T\varphi_1(s)dW_s}
 {\int_0^T\varphi_1^2(s)ds}. \end{equation}
In accordance with Theorem \ref{mixednonrandom}, assume that function $\varphi_1$ satisfies $(B_5)$ and $\varphi_2$ is bounded.
 Then  estimate ${\theta}_T^{(3)}$ is strongly consistent. Evidently, these assumptions hold for the constant coefficients.

\subsection{The fractional Ornstein-Uhlenbeck model and strong consistency.} I. Consider the  fractional Ornstein-Uhlenbeck, or Vasicek, model with non-constant coefficients.  It has a form
$$dX_t = \theta(a(t)X_t+b(t))dt + \gamma(t) dB^H_t,\,t\geq 0,$$ where $a$, $b$ and $\gamma$ are non-random measurable functions. Suppose they are locally bounded and $\gamma=\gamma(t)>0$. The solution for this equation is a Gaussian process and has a form
\begin{equation*} X_t=e^{\theta A(t)}\Big(x_0+\theta\int_0^tb(s)e^{-\theta A(s)}ds+\int_0^t\gamma(s)e^{-\theta A(s)}dB^H_s\Big):=E(t)+G(t),
\end{equation*}
where
$A(t)=\int_0^ta(s)ds$, $E(t)=e^{\theta A(t)}\Big(x_0+\theta\int_0^tb(s)e^{-\theta A(s)}ds\Big)$ is a non-random function, $G(t)=e^{\theta A(t)}\int_0^t\gamma(s)e^{-\theta A(s)}dB^H_s$ is a Gaussian process with zero mean.

 Denote $c(t)=\frac{a(t)}{\gamma(t)}, d(t)=\frac{b(t)}{\gamma(t)}.$ Now  we shall state the conditions for strong consistency of the maximum likelihood estimate.
 \begin{theorem}\label{ornstein} Let  functions $a$, $c$, $d$ and $\gamma$ satisfy the following assumptions:
\begin{enumerate} \item [$(B_7)$] $-a_1\leq a(s)\leq -a_2<0$,  $-c_1\leq c(s)\leq -c_2<0$, $0<\gamma_1\leq\gamma(s)\leq\gamma_2$,  functions $c$ and $d$  are continuously  differentiable, $c'$ is bounded,  $c'(s)\geq 0$ and $c'(s)\rightarrow 0$ as $s\rightarrow\infty$.
\end{enumerate}
Then estimate ${\theta}_T^{(1)}$ is strongly consistent as $T\rightarrow\infty$.
 \end{theorem}

\begin{proof} We shall check the conditions of Proposition \ref{pro_2.1}. Obviously, $\psi(t,x)=c(t)x+d(t)\in C^1(\RR^+)\times C^2(\RR) $ and
 $$J(t)=\int_0^tl_H(t,s)(d(s)+c(s)E(s))ds+\int_0^tl_H(t,s)c(s)G(s)ds:=F(t)+H(t). $$ Furthermore,  assumptions $(A_1)$, $(A_3)$, $(A'_2)$, $(A'_4)$ and $(B_1)$ hold. Note that the trajectories of  process $G$ are a.s. H\"{o}lder up to order $H$, whence  $$\lim_{s\rightarrow0}s^{\frac12-H}c(s)G(s)=0.$$ Therefore \begin{equation*}\begin{gathered}J'(t)=F'(t)+H'(t)=F'(t)+\int_0^tl_H(t,s)f(s)G(s)ds+\int_0^tl_H(t,s)c(s)\gamma(s)dB_s^H,\end{gathered}\end{equation*}
where $f(s)=\Big(\frac12-H\Big)s^{-1}c(s)+c'(s)+\theta a(s)c(s)$. Evidently, $J'_t$ is  Gaussian process with mean and variance that are bounded on any bounded interval. Therefore, condition   $(B_2)$ holds.  As for condition $(B_3)$, we must verify that $I_\infty=\int_0^\infty(J'_t)^2t^{2H-1}dt=\infty$ a.s. For any $\lambda>0$ consider the moment generation function  $$\Theta_T(\lambda)=E \exp\{-\lambda I_T\}=E\exp\{-\lambda\int_0^T (J'_t)^2t^{2H-1}dt\}$$ and $$\Theta_\infty(\lambda)=E \exp\{-\lambda I_\infty\}=E\exp\{-\lambda\int_0^\infty (J'_t)^2t^{2H-1}dt\},$$ so that  $\Theta_\infty(\lambda)=\lim_{T\rightarrow\infty}\Theta_T(\lambda)$. Evidently,
    \begin{eqnarray} \begin{gathered}
    \int_0^T(J'_t)^2t^{2H-1}dt\geq T^{-1} \Big(\int_0^TJ'_tt^{H-\frac12}dt\Big)^2,\nonumber
    \end{gathered}
    \end{eqnarray}
    whence $$\Theta_T(\lambda)\leq \Theta_T^{(1)}(\lambda):=E\exp\Big\{-\frac{ \lambda}{T}\Big(\int_0^TJ'_tt^{H-\frac12}dt\Big)^2\Big\}.$$
     Random  variable $\int_0^TJ'_tt^{H-\frac12}dt$ is Gaussian with  mean $M(T)$ and variance $\sigma^2(T)$, say.
     Note that for a Gaussian random variable $\xi=m+\sigma N(0,1)$ we can easily calculate \begin{equation}\label{sigma}E\exp\{- a\xi^2\}=\Big(2a\sigma^2+1\Big)^{-\frac12}\exp\Big\{-\frac{am^2}{2a\sigma^2+1}\Big\}.\end{equation}
     This value attains its maximum at the point  $m=0$. Hence,  it is sufficient to prove that $$\lim_{T\rightarrow\infty}\Theta_T^{(2)}(\lambda):=\lim_{T\rightarrow\infty}E\exp\Big\{-\frac{ \lambda}{T}\Big(\int_0^TH'_tt^{H-\frac12}dt\Big)^2\Big\}=0.$$
     However, it follows from \eqref{sigma} that $\Theta_T^{(2)}(\lambda)=\Big(\frac{2\lambda\sigma^2_T}{T}+1\Big)^{-\frac12}$, therefore to prove the strong consistency of the maximum likelihood estimate $\theta_T^{(1)}$, we only need to analyze the asymptotic behavior of  $\sigma^2_T$. More specifically, we need to prove that $\frac{\sigma^2_T}{T}\rightarrow \infty$ as $T\rightarrow \infty$.
     In what follows we apply the following  formulae from \cite{nvv} and \cite{mmv} for   Wiener integrals w.r.t. the fractional Brownian motion
    \begin{equation*}\begin{gathered}E\int_0^{t_1}g(s)dB_s^H\int_0^{t_2}h(s)dB_s^H
    =H(2H-1)\int_0^{t_1}\int_0^{t_2}g(s_1)h(s_2)|s_1-s_2|^{2H-2}ds_1ds_2
    \\\leq C(H)||g||_{L_{\frac1H}[0,t_1]}||h||_{L_{\frac1H}[0,t_2]}.\end{gathered}\end{equation*}

     (a) Let $\theta<0$. Divide $\int_0^TH'_tt^{H-\frac12}dt$ into two parts:
     $\int_0^TH'_tt^{H-\frac12}dt=H^{(1)}_T+H^{(2)}_T$, where $$H^{(1)}_T=\int_0^Tt^{H-\frac12}\int_0^tl_H(t,s)f(s)G(s)dsdt$$ and $$H^{(2)}_T=\int_0^Tt^{H-\frac12}\int_0^tl_H(t,s)c(s)\gamma(s)dB_s^Hdt.$$

 Since  functions $c$ and $\gamma$ are bounded from below and from above, \begin{equation}\label{tetaless1}\begin{gathered}E\big(H^{(2)}_T\big)^2=C(H)\int_0^T\int_0^T(t_1t_2)^{H-\frac12}
 \int_0^{t_1}\int_0^{t_2}\Pi_{i=1,2}l_H(t_i,s_i)
     (-c(s_i))\gamma(s_i)\\\times|s_1-s_2|^{2H-2}ds_1ds_2dt_1dt_2\asymp  C(H)\int_0^T\int_0^T(t_1t_2)^{H-\frac12}\\\times\int_0^{t_1}\int_0^{t_2}\Pi_{i=1,2}l_H(t_i,s_i)|s_1-s_2|^{2H-2}ds_1ds_2dt_1dt_2\asymp C(H)T^{3}\end{gathered}\end{equation}
as $T\rightarrow\infty$.

     Consider the behavior of $f$. Under  assumption $(B_7)$  terms $s^{-1}c(s)+c'(s)$ vanish at infinity, $\theta a(s)c(s)$ is negative and separated from zero. Therefore, there exist  $C_i>0$, $i=1,2$ and $s_0 >0$ such  that $-C_1\leq f(s)\leq -C_2$ for all $s>s_0$. Boundedness of $f$ implies that   $E(H_T^{(1)})^2$ has the same asymptotic behavior  as

        \begin{equation}\begin{gathered}\label{tetaless2}
       \int_{s_0}^{T}\int_{s_0}^{T}(t_1t_2)^{H-\frac12}\int_{s_0}^{t_1}\int_{s_0}^{t_2}(\Pi_{i=1,2}l_H(t_i,s_i)(-f(s_i)))\\\times
    \Big(\int_{s_0}^{s_1}\int_{s_0}^{s_2}\gamma(u_1)\gamma(u_2)\exp\Big\{\theta\Big(\int_{u_1}^{s_1}\\+\int_{u_2}^{s_2}\Big)a(v)dv\Big\}
    |u_1-u_2|^{2H-2}du_1du_2\Big)ds_1ds_2dt_1dt_2\\\geq C(H)\int_{s_0}^{T}\int_{s_0}^{T}(t_1t_2)^{H-\frac12}\int_{s_0}^{t_1}\int_{s_0}^{t_2}(\Pi_{i=1,2}l_H(t_i,s_i))\\\times
    \Big(\int_{s_0}^{s_1}\int_{s_0}^{s_2}
    |u_1-u_2|^{2H-2}du_1du_2\Big)ds_1ds_2dt_1dt_2\asymp C(H)T^5.
    \end{gathered}\end{equation}
    Relations \eqref{tetaless1} and \eqref{tetaless2} mean that the asymptotic behavior of $\sigma^2_T$ is $\sigma^2_T\asymp C(H)T^5$ and $\frac{\sigma^2_T}{T}\rightarrow\infty$ as $T\rightarrow\infty$.

    (b) Let $\theta>0$. This case is more involved. The asymptotic behavior of $E(H_T^{(2)})^2$ is the same as before, $C(H)T^3$, since it does not depend on $\theta$. As for $E(H_T^{(1)})^2$, denote $K'_t=\int_0^tl_H(t,s)f(s)G(s)ds$, then  $$H^{(1)}_T=\int_0^TK'_tt^{H-\frac12}dt=T^{H-\frac12}K_T-\Big(H-\frac12\Big)\int_0^Tt^{H-\frac32}K_tdt.$$
    In addition, denote $r(t)=\exp\{-\theta\int_0^ta(s)ds\}$, $\psi(t)=f(t)\exp\{\theta\int_0^ta(s)ds\}$. Applying Fubini theorem several times, we obtain that \begin{equation*}\begin{gathered}T^{H-\frac12}K_T-\Big(H-\frac12\Big)\int_0^Ts^{H-\frac32}K_sds
    \\=T^{H-\frac12}\int_0^Tl_H(T,t)f(t)G(t)dt-\Big(H-\frac12\Big)\int_0^Tt^{H-\frac32}\int_0^tl_H(t,s)f(s)G(s)dsdt\\=
    T^{H-\frac12}\int_0^Tl_H(T,t)\psi(t)\int_0^tr(s)dB_s^Hdt\\-\Big(H-\frac12\Big)\int_0^T
    t^{H-\frac32}\int_0^tl_H(t,u)\psi(u)\int_0^ur(s)dB_s^Hdudt\\=\int_0^Tr(s)\int_s^Tl_H(T,t)\psi(t)dtdB_s^HT^{H-\frac12}-
    \\\Big(H-\frac12\Big)\int_0^Tr(s)\int_s^Tt^{H-\frac32}\int_s^tl_H(t,u)\psi(u)dudtdB^H_s
    \\=\int_0^Tr(s)\Big(T^{H-\frac12}\int_s^Tl_H(T,t)\psi(t)dt
    \\-\Big(H-\frac12\Big)\int_s^Tt^{H-\frac32}\int_s^tl_H(t,u)\psi(u)dudt\Big)dB^H_s.
    \end{gathered}\end{equation*}

   Denote \begin{equation*}\begin{gathered}F(T,s)= T^{H-\frac12}\int_s^Tl_H(T,t)\psi(t)dt
    -\Big(H-\frac12\Big)\int_s^Tt^{H-\frac32}\int_s^tl_H(t,u)\psi(u)dudt\\=\int_s^T
    t^{\frac12-H}e^{\theta\int_0^ta(s)ds}f(t)\Big(T^{H-\frac12}(T-t)^{\frac12-H}\\-\Big(H-\frac12\Big)
    \int_t^Tu^{H-\frac32}(u-t)^{\frac12-H}du\Big)dt:=F_1(T,s)-F_2(T,s).
    \end{gathered}\end{equation*}
    and \begin{equation*}\begin{gathered}F^+(T,s)=F_1(T,s)+F_2(T,s).
    \end{gathered}\end{equation*}
    Function $f$ is bounded,  positive for $s>s_0$ and  separated from zero.  For the sake of technical simplicity, we can put $f(t)=a(t)\equiv 1$. Besides, we can omit the constant multiplier $c_H$.  Then \begin{equation*}\begin{gathered}0\leq  E(H^{(1)}_T)^2=\int_0^T\int_0^Te^{\theta s}e^{\theta t}F(T,s)F(T,t)|s-t|^{2H-2}dsdt\\\leq \int_0^{T}\int_0^{T}e^{\theta s}e^{\theta t}F^+(T,s)F^+(T,t)|s-t|^{2H-2}dsdt.\end{gathered} \end{equation*}
    Consider  the terms containing $F_1(T,s)F_1(T,t)$:
    \begin{equation*}\begin{gathered} I_1=\int_0^T\int_0^Te^{\theta s}e^{\theta t}F_1(T,s)F_1(T,t)|s-t|^{2H-2}dsdt\\=T^{2H-1}\int_0^{T}\int_0^{T}e^{\theta s}e^{\theta t}\int_s^T u^{\frac12-H} (T-u)^{\frac12-H}e^{-\theta u}du\\\times\int_t^T v^{\frac12-H} (T-v)^{\frac12-H}e^{-\theta v}dv|s-t|^{2H-2}dsdt\\
    \leq T^{2H-1}\int_0^{T}\int_0^{T}(st)^{\frac12-H}\int_s^T  (T-u)^{\frac12-H}e^{-\theta (u-s)}du\\\times\int_t^T (T-v)^{\frac12-H}e^{-\theta (v-t)}dv|s-t|^{2H-2}dsdt.\end{gathered} \end{equation*}
    Applying H\"{o}lder inequality we conclude that integral $\int_s^T  (T-u)^{\frac12-H}e^{-\theta (u-s)}du$ admits the following bound:
    \begin{equation*}\begin{gathered} \int_s^T  (T-u)^{\frac12-H}e^{-\theta (u-s)}du
    \\\leq \int_s^{\frac{s+T}{2}}(T-u)^{\frac12-H}e^{-\theta (u-s)}du+\int_{\frac{s+T}{2}}^T(T-u)^{\frac12-H}e^{-\theta (u-s)}du\\\leq 2^{H-\frac12}(T-s)^{\frac12-H}\int_s^{\frac{s+T}{2}}e^{-\theta (u-s)}du+\Big(\int_{\frac{s+T}{2}}^T(T-u)^{1-2H}du\Big)^{\frac12}\Big(\int_{\frac{s+T}{2}}^Te^{-\theta (u-s)}du\Big)^{\frac12}\\\leq C(H)\Big((T-s)^{\frac12-H}+(T-s)^{1-H}\Big).\end{gathered} \end{equation*} Therefore \begin{equation*}\begin{gathered}I_1\leq C(H)\Big(T^{2H-1}\int_0^{T}\int_0^{T}(st)^{\frac12-H}
    ((T-t)(T-s))^{\frac12-H}|s-t|^{2H-2}dsdt\\+T\int_0^{T}\int_0^{T}(st)^{\frac12-H}|s-t|^{2H-2}dsdt\Big)
    \leq C(H)T^2.\end{gathered} \end{equation*}

    Furthermore,  function $e^{\theta s}F_2(T,s)$ admits the following  bounds: \begin{equation*}\begin{gathered}e^{\theta s}F_2(t,s)\leq C(H)s^{\frac12-H}\int_s^Tt^{H-\frac32}(T-t)^{\frac32-H}e^{-\theta(t-s)}dt\\\leq C(H)T^{\frac32-H}s^{\frac12-H}\int_s^Tt^{H-\frac32}e^{-\theta(t-s)}dt. \end{gathered} \end{equation*}
    Note that  function $\int_s^Tt^{H-\frac32}e^{-\theta(t-s)}dt$ decreases in $s$ since its derivative equals $e^{\theta s}(\int_s^Tt^{H-\frac32}e^{-\theta t}dt-s^{H-\frac32})<0.$ Therefore, $$e^{\theta s}F_2(t,s)\leq C(H)T^{\frac32-H}s^{\frac12-H}\int_0^Tt^{H-\frac32}e^{-\theta t}dt\leq C(H)T^{\frac32-H}s^{\frac12-H}.$$
    The latter implies  that the term containing $F_2(T,s)F_2(T,t)$ admits the following bounds: \begin{equation*}\begin{gathered} I_2=\int_0^T\int_0^Te^{\theta s}e^{\theta t}F_2(T,s)F_2(T,t)|s-t|^{2H-2}dsdt
    \\\leq C(H)T^{3-2H}\int_0^T\int_0^T (st)^{\frac12-H}|s-t|^{2H-2}dsdt\leq C(H)T^{4-2H}.\end{gathered} \end{equation*}
    So,    $E(H_T^{(1)})^2\asymp C(H)T^{4-2H}$ asymptotically and if we compare this to asymptotical behavior of $E(H_T^{(2)})^2\asymp C(H)T^3$, we can conclude that $\frac{\sigma^2_T}{T}\asymp C(H)T^2 \rightarrow \infty$ as $T\rightarrow \infty$.

    (c) Let $\theta=0$. Then it is  easy to verify that $E(H_T^{(1)})^2\asymp C(H)T$ and we can refer to the case $\theta>0$.
    \end{proof}
    \begin{remark} The assumptions of the theorem are fulfilled, for example, if  $a(s)=-1$, $b(s)=b\in \mathbb{R}$ and $\gamma(s)=\gamma>0$. In this case we deal with a standard Ornstein-Uhlenbeck process $X$ with constant coefficients that satisfies the equation    $$dX_t = \theta(b-X_t )dt + \gamma dB^H_t,\,t\geq 0.$$
    This model with constant coefficients was studied in \cite{KlLeBr} where the Laplace transform $\Theta_T(\lambda)$ was calculated explicitly and strong consistency of $\theta_T^{(1)}$ was established. Therefore, our results generalize the statement of strong consistency to the case of variable coefficients.
     \end{remark}

     II. Consider a simple version of the Ornstein-Uhlenbeck model where $a=\gamma=1$, $b=x_0=0$. The SDE has a form $dX_t = \theta X_tdt + dB^H_t,\,t\geq 0 $ with evident solution $X_t=e^{\theta t}\int_0^te^{-\theta s}dB_s^H$. Let us construct an estimate
     which is a modification of $\theta_T^{(2)}$:
     \begin{equation*}\begin{gathered}\widetilde{\theta}_T^{(2)}=\frac{\int_0^T e^{-2\theta s} X_sdX_s}{\int_0^T e^{-2\theta s} X_s^2ds}=\theta+\frac{\Big(\int_0^Te^{-\theta s}dB_s^H\Big)^2}{\int_0^T\Big(\int_0^se^{-\theta u}dB_u^H\Big)^2ds}.\end{gathered} \end{equation*}
     \begin{theorem} Let $\theta>0$. Then   estimate ${\widetilde{\theta}}_T^{(2)}$ is strongly consistent as $T\rightarrow\infty$.
     \end{theorem}
\begin{proof} Applying Remark \ref{prosto} yields $$|\int_0^Te^{-\theta s}dB_s^H|\leq e^{-\theta T}|B_T^H|+\int_0^T e^{-\theta s}|B_s^H|ds\leq \xi \Big(e^{-\theta T}T^{H+p}+\int_0^T e^{-\theta s}s^{H+p}ds\Big)\leq \zeta,$$
where $\zeta$ is a random variable independent of $T$. So,
it is sufficient to establish that $\int_0^\infty\Big(\int_0^se^{-\theta u}dB_u^H\Big)^2ds=0$ to prove the strong consistency of $\widetilde{\theta}_T^{(2)}$. Similarly to the proof of Theorem \ref{ornstein}, we can consider the moment generation function \begin{equation*}\begin{gathered}E\exp\{-\lambda\int_0^T\Big(\int_0^se^{-\theta u}dB_u^H\Big)^2ds\}\leq E\exp\Big\{-\lambda T^{-1}\Big(\int_0^T\int_0^se^{-\theta u}dB_u^Hds\Big)^2\Big\}\\=\Big(\frac{2\lambda\sigma^2_T}{T}+1\Big)^{-\frac12},\end{gathered} \end{equation*}
where \begin{equation*}\begin{gathered}\sigma^2_T=E\Big(\int_0^T\int_0^se^{-\theta u}dB_u^Hds\Big)^2=
\int_0^T\int_0^T\int_0^s\int_0^te^{-\theta u-\theta v}|u-v|^{2H-2}dudvdsdt
\\=T^{2H+2}\int_0^1\int_0^1\int_0^s\int_0^te^{-T(\theta u+\theta v)}|u-v|^{2H-2}dudvdsdt
\\\geq T^{2H+2}\int_0^1\int_0^1\int_0^s\int_0^te^{-T(\theta u+\theta v)}dudvdsdt\\=T^{2H}\theta^{-2}\Big(\int_0^1\Big(1-e^{-\theta sT}\Big)ds\Big)^2\asymp T^{2H}\theta^{-2},\end{gathered} \end{equation*}
whence the proof follows.
\end{proof}

\appendix
\section{}

 To apply Theorem \ref{teo2} to the fractional derivative of the fractional Brownian motion and to prove Theorem \ref{teo4}, we need an auxiliary result. In what follows we denote by $C(H,\alpha)$ a constant depending only on $H$ and $\alpha$ and not on  other parameters.

\begin{lemma}\label{lemGaus1}
Let $z_i > 0$ for $i=1,2$. In addition, let $0 <{H}<{1},\, 1-H<\alpha<1$ and
$$
I = {z_2}^{2(H + \alpha - 1)} + {z_1}^{2(H + \alpha - 1)} + \frac{|z_2 - z_1|^{2H} - z_1^{2H} - z_2^{2H}}{(z_1z_2)^{1 - \alpha}}.
$$
Then $I \leq C(H,\alpha)|z_2 - z_1|^{2(H + \alpha - 1)}.$

\end{lemma}

\begin{proof} Let   $z_2>z_1>0$ (the case $z_1>z_2>0$ can be  dealt with in a similar way). We can rewrite $I$ as
\begin{equation*}
\begin{gathered}
I= (z_2^{H+\alpha-1}-z_1^{H+\alpha-1})^2+2(z_1z_2)^{H+\alpha-1}\\+((z_2-z_1)^{2H}-(z_2^H-z_1^H)^{2}-2(z_1z_2)^H)(z_1z_2)^{\alpha-1} \\
=(z_2^{H+\alpha-1}-z_1^{H+\alpha-1})^2+\frac{(z_2-z_1)^{2H}-(z_2^H-z_1^H)^{2}}{(z_1z_2)^{1-\alpha}}=I_1+I_2.
\end{gathered}
\end{equation*}
Recall a  simple inequality $b^{r} - a^{r} \leq (b - a)^{r}$ for $b > a,\ 0< r \leq 1$. Since $0<H + \alpha - 1 < 1$, we can estimate $I_1$ by $( z_2 - z_1)^{2(H + \alpha - 1)}.$
Furthermore,  $I_2$ can be rewritten as
$$
I_2 = (z_2 - z_1)^{2(H + \alpha - 1)}\frac{|z_2 - z_1|^{2H} - (z_2^{H} - z_1^{H})^2}{(z_1z_2)^{1 - \alpha}(z_2 - z_1)^{2(H + \alpha - 1)}} = (z_2 - z_1)^{2(H + \alpha - 1)}f(u),
$$
where $u = \frac{z_2}{z_1} > 1$, $f(u) = \frac{(u - 1)^{2H} - (u^H - 1)^2}{u^{1 - \alpha}(u - 1)^{2(H + \alpha - 1)}} \geq 0.$\\

Calculate the limit of  function $f$ at 1:
$$
\lim_{u\rightarrow 1}\limits f(u) = \lim_{u\rightarrow 1}\limits \frac{(u - 1)^{2H} - (u^H - 1)^2}{(u - 1)^{2(H + \alpha - 1)}}.
$$
Here
$$
\lim_{u\rightarrow 1}\limits \frac{(u - 1)^{2H}}{(u - 1)^{2(H + \alpha - 1)}} = \lim_{u\rightarrow 1}\limits (u - 1)^{2 - 2\alpha} = 0,
$$
and
$$
\lim_{u\rightarrow 1}\limits\frac{(u^H - 1)^{2}}{(u - 1)^{2(H + \alpha - 1)}} = H^2\lim_{u\rightarrow 1}\limits (u - 1)^{4 - 2H - 2\alpha} = 0,
$$
since
$\lim_{u\rightarrow 1}\limits \frac{u^{H}-1}{u-1}=H.$
Calculate the limit of the function $f$ at infinity:
$$
\begin{array}{rcl}
0 &\leq& \lim_{u\rightarrow \infty}\limits f(u) = \lim_{u\rightarrow \infty}\limits \frac{(u - 1)^{2H} - (u^H - 1)^2}{u^{1 - \beta}(u - 1)^{2(H + \alpha - 1)}} \\
&\leq& \lim_{u\rightarrow \infty}\limits \frac{u^{2H} - (u^H - 1)^2}{u^{2H +\alpha- 1}} = \lim_{u\rightarrow \infty}\limits \frac{2u^H - 1}{u^{2H + \alpha - 1}} = 0.
\end{array}
$$
This implies that  function  $f$ is bounded, i.e. there exists $C(H,\alpha)>0$ such that
$$I_2 \leq C(H,\alpha)(z_2 - z_1)^{2(H + \alpha - 1)},$$
and  the proof follows if we combine  the  bounds for $I_1$ and $I_2$.
\end{proof}

 We are now ready to check  conditions $(D_2)$ and $(D_3)$ for the fractional derivative of the fractional Brownian motion.
\begin{lemma}\label{teo1}
Let
$$
X(\textbf{t}) = \frac{B_{t_1}^H - B_{t_2}^H}{({t_1} - {t_2})^{1 - \alpha}} + \int_{{t_2}}^{{t_1}}\frac{B_u^H - B_{t_2}^H}{(u - {t_2})^{2 - \alpha}}du,
$$
where
$
0 \leq{{t_2}}<{{t_1}},  \ 0 <{H}<{1}, 1-H<\alpha<1.
$

Then the following bounds hold:

1) for any $0\leq\ t_2<t_1$ $$\left(E(X(\textbf{t}))^2\right)^{\frac{1}{2}} \leq C(H,\alpha)({t_1} - {t_2})^{H + \alpha - 1};$$

2) (a) Let $H+\alpha\leq\frac32$. Then for any $0\leq\ t_2<t_1$, $0\leq\ s_2<s_1$ and any $0 < \varepsilon < (H + \alpha - 1)\wedge\frac 12$
\begin{equation*}
\begin{gathered}
(E|X(\textbf{t}) - X(\textbf{s})|^2)^{\frac{1}{2}}\\
\leq C(H,\alpha)\big(1 + \varepsilon^{-1}\big)(|t_1 - s_1|\vee|t_2 - s_2|))^{H + \alpha - 1- \varepsilon}(t_1\vee s_1))^{\varepsilon}
\end{gathered}
\end{equation*}
with $C(H,\alpha)$ not depending on $X$, its arguments  and $\varepsilon$.

(b) Let $H+\alpha>\frac32$. Then for any $0\leq\ t_2<t_1$, $0\leq\ s_2<s_1$ \begin{equation*}
\begin{gathered}
(E|X(\textbf{t}) - X(\textbf{s})|^2)^{\frac{1}{2}}
\leq C(H,\alpha)(|t_1 - s_1|\vee|t_2 - s_2|)^{\frac12}(t_1\vee s_1)^{H+\alpha-\frac32}.
\end{gathered}
\end{equation*}
\end{lemma}

\begin{proof} The first statement follows immediately from the Minkowski's integral inequality:

\begin{equation*}\begin{gathered}
\left(E(X(\textbf{t}))^2\right)^{\frac{1}{2}} \leq \left(E \left(\frac{B_{t_1}^H - B_{t_2}^H}{(t_1 - t_2)^{1 - \alpha}} \right)^2\right)^{\frac12} + \left(E \left(\int_{t_2}^{t_1} \frac{B_u^H - B_{t_2}^H}{(u - t_2)^{2 - \alpha}}du \right)^2\right)^{\frac12}\\
\leq \left(\frac{(t_1 - t_2)^{2H}}{(t_1 - t_2)^{2(1 - \alpha)}} \right)^{\frac{1}{2}} + \int_{t_2}^{t_1}\left(E \left(\frac{B_u^H - B_{t_2}^H}{(u - {t_2})^{2 - \alpha}} \right)^2\right)^{\frac{1}{2}}du = (t_1 - t_2)^{H + \alpha - 1} +\\
+ \int_{t_2}^{t_1}\left(\frac{(u - t_2)^{2H}}{(u - t_2)^{2(2 - \alpha)}} \right)^{\frac{1}{2}}du
= \frac{\alpha + H}{\alpha + H - 1}(t_1 - t_2)^{H + \alpha - 1}.
\end{gathered}
\end{equation*}
In order to prove the second statement, denote
$
X_1(\textbf{t}) = \frac{B_{t_1}^H - B_{t_2}^H}{(t_1 - t_2)^{1 - \alpha}}$ and $ X_2(\textbf{t})= \int_{t_2}^{t_1}\frac{B_u^H - B_{t_2}^H}{(u - {t_2})^{2 - \alpha}}du.
$
Evidently,   
\begin{equation}\label{5}
(E|X(\textbf{t}) - X(\textbf{s})|^2)^{\frac{1}{2}} \leq (E|X_1(\textbf{t}) - X_1(\textbf{s})|^2)^{\frac{1}{2}}  \\
+(E|X_2(\textbf{t}) - X_2(\textbf{s})|^2)^{\frac{1}{2}}.
\end{equation}
Let $t_1 > s_1$, the opposite case can be considered in a similar way. Then
\begin{equation}\label{6}
\begin{gathered}
(E|X_1(\textbf{t}) - X_1(\textbf{s})|^2)^{\frac{1}{2}} \\
= \left(E \left(\frac{B_{t_1}^H - B_{t_2}^H}{(t_1 - t_2)^{1 - \alpha}} - \frac{B_{t_1}^H - B_{s_2}^H}{(t_1 - s_2)^{1 - \alpha}} + \frac{B_{t_1}^H - B_{s_2}^H}{(t_1 - s_2)^{1 - \alpha}} - \frac{B_{s_1}^H - B_{s_2}^H}{(s_1 - s_2)^{1 - \alpha}} \right)^2 \right)^{\frac{1}{2}}\\
\leq \left(E \left(\frac{B_{t_1}^H - B_{t_2}^H}{(t_1 - t_2)^{1 - \alpha}} - \frac{B_{t_1}^H - B_{s_2}^H}{(t_1 - s_2)^{1 - \alpha}} \right)^2 \right)^{\frac{1}{2}}\\
+ \left(E \left(\frac{B_{t_1}^H - B_{s_2}^H}{(t_1 - s_2)^{1 - \alpha}} - \frac{B_{s_1}^H - B_{s_2}^H}{(s_1 - s_2)^{1 - \alpha}} \right)^2 \right)^{\frac{1}{2}} =: I_3 + I_4.
\end{gathered}
\end{equation}
It is more convenient to estimate the squares $(I_3)^2$ and $(I_4)^2$ from \eqref{6} instead of $ I_3 $  and $ I_4 $. As for  $(I_3)^2$, we can calculate it explicitly and then estimate it with the help of Lemma \ref{lemGaus1};  $(I_4)^2$ can be evaluated similarly.
\begin{equation}\label{8}
\begin{gathered}
(I_3)^2 = (t_1 - t_2)^{2(H + \alpha - 1)}+(t_1 - s_2)^{2(H + \alpha - 1)} - 2\frac{E(B_{t_1}^H - B_{t_2}^H)(B_{t_1}^H - B_{s_2}^H)}{(t_1 - t_2)^{1 - \alpha}(t_1 - s_2)^{1 - \alpha}}\\
 =(t_1 - t_2)^{2(H + \alpha - 1)} + (t_1 - s_2)^{2(H + \alpha - 1)} - \frac{2}{(t_1 - t_2)^{1 - \alpha}(t_1 - s_2)^{1 - \alpha}}\\
 \times[t_1^{2H}-\frac{1}{2}\left(t_2^{2H} + t_1^{2H} - (t_1 - t_2)^{2H}\right)-\frac{1}{2}\left(t_1^{2H}+s_2^{2H}-(t_1-s_2)^{2H}\right)\\
+ \frac{1}{2}\left(t_2^{2H} + s_2^{2H} - |t_2 - s_2|^{2H}\right)] =
 (t_1 - t_2)^{2(H + \alpha - 1)} + (t_1 - s_2)^{2(H + \alpha - 1)} \\
 +\frac{|t_2 - s_2|^{2H} - (t_1 - t_2)^{2H} - (t_1 - s_2)^{2H}}{(t_1 - t_2)^{1 -  \alpha}(t_1 - s_2)^{1 - \alpha}}\leq C(H,\alpha)|t_2- s_2|^{2(H + \alpha - 1)}.
\end{gathered}
\end{equation}
We derive from \eqref{8} that
\begin{equation}\label{riv_a}
I_3 \leq C(H,\alpha)|t_2 - s_2|^{H + \alpha - 1},
\end{equation}
and  similarly,
\begin{equation}\label{riv_b}
I_4 \leq C(H,\alpha)|t_1 - s_1|^{H + \alpha - 1}.
\end{equation}
It follows immediately from  \eqref{riv_a} and \eqref{riv_b} that
\begin{equation}\label{riv_c}
(E|X_1(\textbf{t}) - X_1(\textbf{s})|^2)^{\frac{1}{2}} \leq  C(H,\alpha)\left(|t_1 - s_1|\vee|t_2 - s_2|\right)^{H + \alpha - 1}.
\end{equation}
Now estimate
\begin{equation*}
F(\textbf{t},\textbf{s}) = (E|X_2(\textbf{t}) - X_2(\textbf{s})|^2)^{\frac{1}{2}}\\ =\left(E\left(\int_{t_2}^{t_1}\frac{B_u^H - B_{t_2}^H}{(u - t_2)^{2 - \alpha}}du - \int_{s_2}^{s_1}\frac{B_u^H - B_{s_2}^H}{(u - s_2)^{2 - \alpha}}du\right)^2\right)^\frac{1}{2}.
\end{equation*}
Let, for instance,  $0\leq t_2 < s_2 < s_1 < t_1$ (other types of relation between these points can be handled  similarly). Then
\begin{equation}\label{7}
\begin{gathered}
F(\textbf{t},\textbf{s})\leq \left(E \left(\int_{t_2}^{s_2}\frac{B_u^H - B_{t_2}^H}{(u - t_2)^{2 - \alpha}}du \right)^2 \right)^{\frac{1}{2}}\\
+ \left(E \left(\int_{s_2}^{s_1} \left(\frac{B_u^H - B_{t_2}^H}{(u - t_2)^{2 - \alpha}} - \frac{B_u^H - B_{s_2}^H}{(u - s_2)^{2 - \alpha}} \right)du \right)^2 \right)^{\frac{1}{2}}\\
+ \left(E \left(\int_{s_1}^{t_1}\frac{B_u^H - B_{t_2}^H}{(u - t_2)^{2 - \alpha}}du \right)^2\right)^{\frac{1}{2}} =: I_5 + I_6 + I_7.
\end{gathered}
\end{equation}
Using the Minkowski's integral inequality we  immediately obtain
\begin{equation}\label{11a}
\begin{gathered}
I_5 \leq \int_{t_2}^{s_2}\left(E\left(\frac{B_u^H - B_{t_2}^H}{(u - t_2)^{2 - \alpha}}\right)^2\right)^{\frac{1}{2}}du
\\= \int_{t_2}^{s_2}(u - t_2)^{H + \alpha - 2}du = \frac{1}{H + \alpha - 1}(s_2 - t_2)^{H + \alpha - 1}.
\end{gathered}
\end{equation}
Similarly,
\begin{equation}\label{11b}
I_7 \leq \frac{1}{H + \alpha - 1}(t_1 - s_1)^{H + \alpha -1}.
\end{equation}

 Again, using the Minkowski's integral inequality and Lemma \ref{lemGaus1} we conclude that

\begin{equation}\label{11c}\begin{gathered}I_6 \leq \int_{s_2}^{s_1} \left(E \left(\frac{B_u^H - B_{t_2}^H}{(u - t_2)^{2 - \alpha}} - \frac{B_u^H - B_{s_2}^H}{(u - s_2)^{2 - \alpha}} \right)^2 \right)^{\frac{1}{2}}du\\
= \int_{s_2}^{s_1} \Bigg[(u - t_2)^{2(H + \alpha - 2)} + (u - s_2)^{2(H + \alpha - 2)}
\\+ \frac{(s_2 - t_2)^{2H} - (u - t_2)^{2H} - (u - s_2)^{2H}}{(u - t_2)^{2 - \alpha}(u - s_2)^{2 - \alpha}} \Bigg]^{\frac{1}{2}}du\\
= \int_{s_2}^{s_1}(u - s_2)^{-{\frac{1}{2}}}(u - t_2)^{-{\frac{1}{2}}}\Bigg[(u - t_2)^{2(H + \alpha - 2)}(u - s_2)(u - t_2)\\
+ (u - s_2)^{2(H + \alpha - 2)}(u - s_2)(u - t_2)
\\+ \frac{(s_2 - t_2)^{2H} - (u - t_2)^{2H} - (u - s_2)^{2H}}{(u - t_2)^{1- \alpha}(u - s_2)^{1 - \alpha}}\Bigg]^{\frac{1}{2}}du \\\leq \int_{s_2}^{s_1}(u - s_2)^{-{\frac{1}{2}}}(u - t_2)^{-{\frac{1}{2}}}\Bigg[(u - t_2)^{2(H + \alpha - 1)} + (u - s_2)^{2(H + \alpha - 1)}\\
+ (u - s_2)^{2(H + \alpha - 2)+1}(s_2-t_2)
+ \frac{(t_2 - s_2)^{2H} - (u - t_2)^{2H} - (u - s_2)^{2H}}{(u - t_2)^{1- \alpha}(u - s_2)^{1 - \alpha}}\Bigg ]^{\frac{1}{2}}du \\
\leq C(H,\alpha)\int_{s_2}^{s_1}(u - s_2)^{-{\frac{1}{2}}}(u - t_2)^{-{\frac{1}{2}}}(s_2 - t_2)^{H + \alpha - 1}du \\
+ C(H,\alpha)\int_{s_2}^{s_1}(u - s_2)^{H + \alpha - 2}(u - t_2)^{-{\frac{1}{2}}}(s_2 - t_2)^{\frac{1}{2}}du =: I_8 + I_9.
\end{gathered}
\end{equation}
Evidently,  $$I_8 =(s_2 - t_2)^{H + \alpha - 1}\int_{s_2}^{s_1}(u - s_2)^{-{\frac{1}{2}}}(u - t_2)^{-{\frac{1}{2}}}du=(s_2 - t_2)^{H + \alpha - 1}I_{10} $$ up to the constant multiplier and for any $0<\varepsilon<\frac12$  integral $I_{10}$  can be rewritten as
\begin{multline*}
I_{10}= \int_{s_2}^{s_1}(u - s_2)^{-{\frac{1}{2}}}(u - t_2)^{-{\frac{1}{2}}}du \\
= \int_{0}^{\frac{s_1 - s_2}{s_2 - t_2}}(y + 1)^{-{\frac{1}{2}}}y^{-{\frac{1}{2}}}dy
\leq \left(\frac{s_1 - s_2}{s_2 - t_2} \right)^\varepsilon\int_{0}^{\frac{s_1 - s_2}{s_2 - t_2}}(y + 1)^{-{\frac{1}{2}}}y^{-{\frac{1}{2}}-\varepsilon}dy\\ \leq \left(\frac{s_1 - s_2}{s_2 - t_2} \right)^\varepsilon\int_{0}^{\infty}(y + 1)^{-{\frac{1}{2}}}y^{-{\frac{1}{2}}-\varepsilon}dy\leq C\big(1 + \varepsilon^{-1}\big)\left(\frac{s_1 - s_2}{s_2 - t_2} \right)^\varepsilon.
\end{multline*}
Therefore, for any  $0<\varepsilon<(H+\alpha-1)\wedge\frac12$
\begin{equation}\label{I10}
I_8 \leq C(H,\alpha)\big(1 + \varepsilon^{-1}\big)(s_2 - t_2)^{H + \alpha - 1 -\varepsilon}(s_1-s_2)^\varepsilon.
\end{equation}
Furthermore,  $$I_9 =(s_2 - t_2)^{\frac12}\int_{s_2}^{s_1}(u - s_2)^{H + \alpha - 2}(u - t_2)^{-{\frac{1}{2}}}du=(s_2 - t_2)^{\frac12}I_{11} $$ up to a constant multiplier. In the case when $H+\alpha<\frac32$   the integral  $I_{11}$ can be rewritten as
\begin{multline*}
I_{11} = \int_{s_2}^{s_1}(u - s_2)^{H + \alpha - 2}(u - t_2)^{-{\frac{1}{2}}}du \\
= \int_{0}^{\frac{s_1 - s_2}{s_2 - t_2}}y^{H + \alpha - 2}(1 + y)^{-{\frac{1}{2}}}(s_2 - t_2)^{H + \alpha - 2 + {\frac{1}{2}}}du \\\leq (s_2 - t_2)^{H + \alpha - {\frac{3}{2}}}\int_{0}^{\infty}y^{H + \alpha - 2}(1 + y)^{-{\frac{1}{2}}}du\leq C(H,\alpha)(s_2 - t_2)^{H + \alpha - {\frac{3}{2}}}.
\end{multline*}
In  case when  $H+\alpha>\frac32$   integral  $I_{11}$ admits an obvious bound
\begin{equation*}
I_{11}\leq\int_{s_2}^{s_1}(u - s_2)^{H + \alpha - 2}(u - s_2)^{-{\frac{1}{2}}}du\leq C(H,\alpha)(s_1-s_2)^{H+\alpha-\frac32}.
\end{equation*}
Finally, for $H+\alpha=\frac32$  integral $I_{11}$ admits the same bound as $I_{10}$.
Therefore, \begin{equation}\label{12f}I_9 \leq C(H,\alpha)(s_2 - t_2)^{H + \alpha - 1}\end{equation} for $H+\alpha<\frac32$, \begin{equation}I_9 \leq C(H,\alpha)(s_2 - t_2)^{\frac12}(s_1-s_2)^{H+\alpha-\frac32}\end{equation} for $H+\alpha>\frac32$, and
\begin{equation}I_9 \leq C(H,\alpha)(s_2 - t_2)^{\frac12 -\varepsilon}(s_1-s_2)^{\varepsilon}\end{equation}
for $H+\alpha=\frac32.$

This implies that
\begin{equation}\label{riv_d<}
F(\textbf{t},\textbf{s})
\leq C(H,\alpha)\big(1 + \varepsilon^{-1}\big)(|t_1 - s_1|\vee|t_2 -s_2|)^{H + \alpha - 1- \varepsilon}(s_1\vee t_1)^{\varepsilon}
\end{equation}
for $H+\alpha\leq\frac32$. In  case $H+\alpha>\frac32$ we can put $\varepsilon=H+\alpha-\frac32\in(0,\frac12)$ in \eqref{I10} and conclude that
\begin{equation}\label{riv_d>}
F(\textbf{t},\textbf{s})
\leq C(H,\alpha)(|t_1 - s_1|\vee|t_2 - s_2|)^{\frac12}(s_1\vee t_1)^{H+\alpha-\frac32}.
\end{equation}

The proof follows immediately from \eqref{5} and (\ref{riv_c})-(\ref{riv_d>}).
\end{proof}

Proof of Theorem \ref{teo4}:
First of all we should verify conditions $(D_1)-(D_3)$. Condition $(D_1)$ is evident, since $X$ is continuous in both  variables. According to the 2nd statement of  Theorem \ref{teo1}, condition $(D_2)$ holds with $\beta=\varepsilon$, $0<\varepsilon<(H + \alpha - 1)\wedge\frac 12$ and $\gamma=H+\alpha-1-\varepsilon$   in  case when $\alpha+H\leq \frac32$, and with $\beta=H+\alpha-\frac32$ and $\gamma=\frac12$ in case when $\alpha+H>\frac32$. According to the first statement of  Theorem \ref{teo1}, condition $(D_3)$ holds with $\delta=H+\alpha-1$.

 Let $A(t)=(t^{H+\alpha-1}|\log t|^p)\vee 1$ for some $p>1$ and for any $t>0$ and let $b_l=e^l,\,l\geq 0$. Then $\delta_l=(e^{l(H+\alpha-1)}l^p)\vee1$ and $A(b_l)=e^{l(H+\alpha-1)}$. Therefore, in this case  series $S(\delta)$ converges  since
$$S(\delta)=e^{H+\alpha-1}+\sum_{l=1}^{\infty}\frac{e^{(l+1)(H+\alpha-1)}}{e^{l(H+\alpha-1)}l^p}=e^{H+\alpha-1}(1+\sum_{l=1}^{\infty}l^{-p})<\infty.$$
Moreover, it is easy to check that $1 + \frac{\beta}{\gamma} - \frac{\delta}{\gamma}=0$ for any values of $\alpha+H$, hence $\kappa_1=0$. This implies  that all   conditions of   Theorem \ref{teo2} hold true and we can apply the  theorem with $A(t)=(t^{H+\alpha-1}|\log t|^p)\vee 1$ which concludes the proof.\hfill$\Box$

\begin{remark}\label{prosto} Instead of the fractional derivative, we can consider the fractional Brownian motion $B^H_t$  itself and apply   the same reasoning to it. This case is much simpler and we immediately obtain that  $\sup_{0\leq s\leq t}|B^H_s|\leq ((t^{H}(\log(t))^{p})\vee1)\xi(p)$ for any $p>1$.
\end{remark}

\textbf{Acknowledgments} This paper was partially supported by NSERC grant 261855. We are thankful to Ivan Smirnov for the assistance in the preparation of the manuscript.

\end{document}